\documentclass[11pt]{article}
\usepackage[T1]{fontenc}
\usepackage[utf8]{inputenc}
\usepackage{ntheorem}
\usepackage{amssymb}
\usepackage[tbtags]{amsmath}
\usepackage{comment}
\usepackage{thm-restate}
\usepackage{latexsym}
\usepackage{url}
\usepackage[affil-it]{authblk}
\usepackage{hyperref}
\usepackage[noabbrev,capitalise]{cleveref}
\usepackage[demo]{graphicx}
\usepackage{tikz}
\usetikzlibrary{shadows}
\usepackage{subcaption}
\usepackage{caption}
\usepackage{float}
\usepackage{xcolor}
\usepackage[normalem]{ulem}

\usepackage{enumitem}
\usepackage{setspace}

\usepackage{mathrsfs}
\usepackage{wasysym}
\usepackage{fancyhdr}
\usepackage{graphicx}
\usepackage[numbers,sort&compress]{natbib}
\usepackage{calligra}

\definecolor{r1}{HTML}{FF8674}
\definecolor{b1}{HTML}{17ABDD}
\definecolor{p1}{HTML}{D4B6D6}
\definecolor{g1}{HTML}{70E2CB}
\definecolor{o1}{HTML}{DFA743}

\usetikzlibrary{calc,decorations.pathreplacing,decorations.markings,shapes.geometric}
\tikzstyle{vertice}=[circle,fill,inner sep=1.4pt]
\tikzset{
	on each segment/.style={
		decorate,
		decoration={
			show path construction,
			moveto code={},
			lineto code={
				\path [#1]
				(\tikzinputsegmentfirst) -- (\tikzinputsegmentlast);
			},
			curveto code={
				\path [#1] (\tikzinputsegmentfirst)
				.. controls
				(\tikzinputsegmentsupporta) and (\tikzinputsegmentsupportb)
				..
				(\tikzinputsegmentlast);
			},
			closepath code={
				\path [#1]
				(\tikzinputsegmentfirst) -- (\tikzinputsegmentlast);
			},
		},
	},
	mid arrow/.style={postaction={decorate,decoration={
				markings,
				mark=at position .6 with {\arrow[scale=1]{stealth}}
	}}},
	lid arrow/.style={postaction={decorate,decoration={
				markings,
				mark=at position .4 with {\arrow[scale=1]{stealth}}
	}}},
	llid arrow/.style={postaction={decorate,decoration={
				markings,
				mark=at position .53 with {\arrow[scale=1]{stealth}}
	}}},
	rid arrow/.style={postaction={decorate,decoration={
				markings,
				mark=at position .7 with {\arrow[scale=1]{stealth}}
	}}},
}

\usepackage{indentfirst}
\usepackage[verbose,tmargin=20mm,bmargin=20mm,lmargin=25mm,rmargin=25mm,headsep=20mm]{geometry}

\theorembodyfont{\normalfont}
\theoremseparator{}
\theoremstyle{plain}

\newtheorem{thm}{Theorem}[section]
\newtheorem{lem}[thm]{Lemma}

\newtheorem{conj}[thm]{Conjecture}

\newtheorem{obser}[thm]{Observation}

{\noindent \emph{Proof.} {}{}}{\hfill
	$\square$\vspace{1em}}

{\noindent \emph{Proof of claim $1$.} {}{#1}{}}{\hfill
	$\Diamond$\vspace{1em}}

\newlist{case}{enumerate}{1}
\setlist[case]{
	label={\textit{Case \arabic*}.}, 
	leftmargin=*,
	font=\itshape,
	itemsep=6pt,
	topsep=6pt
}

\newlist{subcase}{enumerate}{2}
\setlist[subcase]{
	label={\textit{Subcase \arabic{case}.\arabic*}.}, 
	leftmargin=*,
	font=\itshape,
	itemsep=3pt
}

\title{Feedback edge set in bipartite digraph}
\author{Bin Chen$^a$ \hskip 1cm Jianfeng Hou$^b$ \hskip 1cm  Siyue Liu$^b$}

\affil{
	{ \small $^a${School of Mathematics and Statistics, Fuzhou University, Fujian, China}}

{\small$^b$  {Center for Discrete Mathematics, Fuzhou University, Fujian, China}}
	
}

\date{}

\begin{document}
	
	
	\maketitle
	\begin{abstract}
		\quad Let \(\beta(G)\) denote the minimum size of a feedback edge set of a digraph \(G\), and let \(\gamma(G)\) denote the number of unordered pairs of nonadjacent vertices. Motivated by the Chudnovsky--Seymour--Sullivan conjecture for \(3\)-free digraphs, we study the corresponding feedback-edge problem for bipartite digraphs. In the bipartite setting, \(\gamma(G)\) is taken to count only nonadjacent pairs with ends in distinct partite sets. We prove that every \(4\)-free bipartite digraph \(G\) satisfies \(\beta(G)\le \gamma(G)/2\). We also determine the exact Tur\'an number of \(2k\)-free strong bipartite digraphs with partite sets \(X\) and \(Y\): if \(|X|,|Y|\ge k+1\), then the maximum number of edges is
		$$
		(|X|-(k-1))(|Y|-(k-1))+2k-2.
		$$
		Finally, for the extremal case \(k=2\), we analyze the structure of \(4\)-free strong bipartite Tur\'an digraphs and prove the sharper bound \(\beta(G)\le \gamma(G)/3\) for all such digraphs. This constant is attained by a natural balanced three-block construction.
	\end{abstract}

	{\bf Keywords}: {Bipartite digraph; feedback edge set; directed cycle; Tur\'{a}n number.
		
		\baselineskip 15pt
		
		\section{Introduction}
		\label{sec:1}

Throughout this paper, all digraphs are finite, loopless, and have no parallel edges; opposite edges are allowed. Let \(G=(V,E)\) be a digraph. A \emph{directed \(k\)-cycle} is a directed cycle of length \(k\), and \(G\) is \emph{\(k\)-free} if it contains no directed cycle of length at most \(k\). A digraph is \emph{acyclic} if it contains no directed cycle, and is \emph{strong} if every vertex is reachable from every other vertex by a directed path.

Following \cite{chudnovsky2008cycles}, let \(\beta(G)\) be the minimum size of an edge set \(X\subseteq E(G)\) such that \(G\backslash X\) is acyclic, and let \(\gamma(G)\) be the number of unordered pairs of nonadjacent vertices in \(G\). Thus \(\beta(G)\) is the feedback edge number. Equivalently, \( |E(G)|-\beta(G)\) is the maximum number of edges in an acyclic spanning subdigraph of \(G\). The computation of this parameter is one of Karp's classical NP-complete problems \cite{karp1972reducibility}; from the extremal point of view, recent work of Fox, Himwich and Mani \cite{fox2024extremal} studies feedback edge sets in dense digraphs.

This paper concerns upper bounds on \(\beta(G)\) in terms of missing adjacencies under forbidden short-cycle assumptions. Chudnovsky, Seymour and Sullivan \cite{chudnovsky2008cycles} proved that every \(3\)-free digraph \(G\) satisfies \(\beta(G)\le \gamma(G)\), and conjectured the following sharp improvement.

\begin{conj}\label{1.1}\textnormal{(\cite{chudnovsky2008cycles})}
			If $G$ is a $3$-free digraph, then $\beta(G) \leq \frac{1}{2}\gamma(G)$.
		\end{conj}

Conjecture~\ref{1.1} remains open. In 2011, Dunkum, Hamburger and P\'or \cite{dunkum2011destroying} reduced the upper bound to $\beta(G) \leq 0.88\gamma(G)$. The best known general bound is \(\beta(G)\le 0.8616\,\gamma(G)\), due to Chen, Karson, Liu and Shen \cite{chen2015chudnovsky}. One reason for studying such feedback-edge inequalities is their connection with minimum out-degree conditions for short directed cycles. Chudnovsky et al. pointed out that progress on Conjecture~\ref{1.1} would improve known bounds toward the following conjecture of Caccetta and H\"aggkvist.

		\begin{conj}\label{1.2}\textnormal{(\cite{caccetta1978minimal})}
			Every digraph on $n$ vertices with minimum out-degree at least $n/3$ contains a directed triangle.
		\end{conj}

Indeed, even the weaker estimate \(\beta(G)\le\gamma(G)\) has consequences for Conjecture~\ref{1.2}: Hamburger, Haxell and Kostochka \cite{hamburger2007directed} used it to prove that every digraph on \(n\) vertices with minimum out-degree at least \(0.3532n\) contains a directed triangle, improving earlier thresholds of Caccetta and H\"aggkvist \((0.3820n)\), Bondy \cite{bondy} \((0.3798n)\), and Shen \cite{shen1998directed} \((0.3543n)\). Hladk{\'y}, Kr{\'a}l' and Norin \cite{hladky2009counting} later combined feedback-edge estimates with flag algebras to lower the threshold to \(0.3465n\). These applications show that bounds for \(\beta(G)\) under forbidden-cycle hypotheses are closely related to classical problems on short directed cycles.

We develop a bipartite analogue of this problem. In a bipartite digraph every directed cycle has even length, and hence the first short-cycle condition not forced by the bipartition is the absence of directed \(4\)-cycles, in addition to directed \(2\)-cycles. Bipartite short-cycle problems have recently been studied from a minimum out-degree perspective by Seymour and Spirkl \cite{seymour2020short}, which suggests that the bipartite setting has extremal features distinct from the non-bipartite case. The parameter \(\gamma(G)\) also requires a convention adapted to bipartite graphs: if \(G\) has partite sets \(X\) and \(Y\), then throughout the bipartite part of the paper \(\gamma(G)\) denotes the number of unordered nonadjacent pairs \(\{x,y\}\) with \(x\in X\) and \(y\in Y\). Nonadjacencies inside a partite set are automatic and are not counted.

Our first main result gives the corresponding feedback-edge estimate for \(4\)-free bipartite digraphs.

		\begin{thm}\label{1/2}
	If $G$ is a $4$-free bipartite digraph, then $\beta(G) \le \frac{1}{2}\gamma(G)$.
\end{thm}

Theorem~\ref{1/2} may be viewed as a bipartite counterpart of the Chudnovsky--Seymour--Sullivan estimate. The proof is inductive. Its main step selects a vertex by double-counting induced directed paths of length three and then decomposes the digraph according to the in- and out-neighborhoods of that vertex. The absence of directed \(4\)-cycles is used to separate the resulting parts and to compare the edges that must be deleted with the cross-part nonadjacencies that are created.

The second part of the paper treats the dense extremal case. Since dense bipartite digraphs have few cross-part nonadjacent pairs, they are a natural class in which to test the strength of inequalities relating \(\beta(G)\) and \(\gamma(G)\). We therefore consider a directed Tur\'an problem for \(2k\)-free strong bipartite digraphs. Given a digraph \(G\) and a family of digraphs \(\mathcal{F}\), we say that \(G\) is \emph{\(\mathcal{F}\)-free} if it contains no subdigraph isomorphic to any member of \(\mathcal{F}\). With the relevant vertex set fixed, the \emph{Tur\'an number} of \(\mathcal{F}\) is the maximum number of edges in an \(\mathcal{F}\)-free digraph, and extremal digraphs attaining this maximum are called \emph{Tur\'an digraphs}.

A classical directed Tur\'an-type problem of Thomassen (see \cite{1}, Chapter \(8\)) asks for the least number \(m(n,k)\) such that every strong digraph on \(n\) vertices with at least \(m(n,k)\) edges contains a directed cycle of length at most \(k\). Bermond, Germa, Heydemann and Sotteau \cite{2} completely settled this problem. Equivalently, they proved the following theorem.

 \begin{thm} \textnormal{(\cite{2})}\label{turank}
The Tur\'{a}n number of $k$-free strong digraphs on $n$ vertices is $\binom{n-k+2}{2}+k-2$. 
\end{thm}

Our second main result is an exact bipartite analogue with prescribed part sizes.

 \begin{thm}\label{turan2k}
The Tur\'an number of $2k$-free strong bipartite digraphs with partite sets
$X$ and $Y$, where $\min\{|X|,|Y|\}\ge k+1$, is
$$(|X|-(k-1))(|Y|-(k-1))+2k-2.$$
\end{thm}

For \(k=2\), Theorem~\ref{turan2k} gives the maximum edge number \((|X|-1)(|Y|-1)+2\). Thus a \(4\)-free strong bipartite Tur\'an digraph has \(\gamma(G)=|X|+|Y|-3\). We further analyze the structure of these extremal digraphs and obtain a sharper feedback-edge estimate on this dense class.

\begin{thm}\label{C_{2k}}
	Let $G$ be any $4$-free strong bipartite Tur\'{a}n digraph. Then $\beta(G) \le \frac{1}{3}\gamma(G)$.
\end{thm}

The coefficient \(1/3\) is best possible within this extremal class: equality is attained by balanced members of the three-block family described in Section~\ref{sec:proof C_4}. Thus Theorem~\ref{1/2} gives a general bipartite feedback-edge bound, while Theorem~\ref{C_{2k}} verifies the sharper coefficient for the densest strong \(4\)-free bipartite digraphs.

The rest of this paper is organized as follows. Section~2 collects notation and preliminary observations. Section~3 proves the general bound in Theorem~\ref{1/2}. Section~4 determines the bipartite Tur\'an number in Theorem~\ref{turan2k}. Section~5 proves Theorem~\ref{C_{2k}} by first describing the relevant \(4\)-free strong bipartite Tur\'an digraphs and then applying that structure to feedback edge sets. Section~6 contains concluding remarks.

		\section{Preliminaries}

 For integers $m$ and $n$, write $[m,n]=\{m, m+1, \ldots, n\}$ and $[n]=[1,n]$. Let $G=(V,E)$ be a digraph. We write $(x,y)$ or $xy$ for an edge from $x$ to $y$, and sometimes we will also write $x\rightarrow y$ for this edge and say $x$ dominates $y$. Two distinct vertices are \emph{adjacent} if there is an edge connecting them, otherwise they are nonadjacent. We use $x\sim y$ to mean that $x$ and $y$ are adjacent, and $x\nsim y$ otherwise. The \emph{out-neighborhood} of $x$ is the set $N^{+}_{G}(x)=\{y\in V(G):xy\in E(G)\}$, and its \emph{out-degree} is $d^{+}_{G}(x)=|N^{+}_{G}(x)|$. The \emph{in-neighborhood} $N^{-}_{G}(x)$ and the \emph{in-degree} $d^{-}_{G}(x)$ are defined analogously. The \emph{neighborhood} of $x$ is the set $N_{G}(x)=N^{+}_{G}(x)\cup N^{-}_{G}(x)$ and its \emph{degree} is $d_{G}(x)=d^{+}_{G}(x)+d^{-}_{G}(x)$. The vertices in $N^{+}_{G}(x),N^{-}_{G}(x)$ and $N_{G}(x)$ are called the \emph{out-neighbors}, \emph{in-neighbors} and \emph{neighbors} of $x$, respectively. We denote by $\Delta^{+}(G)$, $\Delta^{-}(G)$ and $\Delta(G)$ the \emph{maximum out-degree}, \emph{maximum in-degree}) and \emph{maximum degree} of $G$, respectively. Also, we use $\delta^{+}(G)$, $\delta^{-}(G)$ and $\delta(G)$  to denote the corresponding minimum ones.

A \emph{directed path} of $G$ is a sequence of distinct vertices $x_{1},x_{2},\ldots,x_{\ell}$ with $x_ix_{i+1}\in E(G)$ for any $i\in[\ell-1]$. The vertices $x_{1}$ and $x_\ell$ are called the \emph{initial} and \emph{terminal} vertices, and the directed path is denoted by $x_1x_2\ldots x_\ell$. Let $P$ and $Q$ be two vertex disjoint directed paths such that the initial vertex of $Q$ is dominated by the terminal vertex of $P$. Write $P\circ Q$ for the concatenation of $P$ and $Q$, i.e., the directed path obtained by first passing through every vertex of $P$ and then those of $Q$. In particular, when $P$\;(resp., $Q$) consists of an edge $(x,y)$, we shall write $(x,y)\circ Q$\;(resp., $P \circ (x,y)$) for $P\circ Q$, and similarly when $P$\;(resp., $Q$) consists of a single vertex $x$, we shall write $x\circ Q$\;(resp., $P \circ x$) for $P\circ Q$. A \emph{directed cycle} of $G$ is a sequence of distinct vertices $y_{1},y_{2},\ldots,y_{h}$ with $y_{i}y_{i+1}\in E(G)$ for any $i\in [h-1]$ and $y_{h}y_{1}\in E(G)$; we denote it by $\langle y_{1}, y_{2},\ldots,y_{h}\rangle$ and write $|C|$ for its length. A digraph is \emph{strongly connected}\;(\emph{strong} for short) if every vertex can be reached from every other vertex via a directed path. A maximal strong subdigraph of $G$ is called a \emph{strong component} of $G$. If $G$ has $s\geq 2$ strong components, then we can arrange these $s$ strong components, denoted by $G_{1},G_2,\ldots,G_{s}$, such that there is no edge from $V(G_j)$ to $V(G_i)$, where $i,j\in [s]$ and $i<j$. In such an ordering, $G_{1}$ is an \emph{initial component} and $G_{s}$ is a \emph{terminal component}.

A digraph $G^*$ is a \emph{subdigraph} of $G$ if $V(G^*)\subseteq V(G)$ and $E(G^*)\subseteq E(G)$. We denote $N^{+}_{G^*}(x)=N^{+}_{G}(x)\cap V(G^*)$ and $d^{+}_{G^*}(x)=|N^{+}_{G^*}(x)|$; the notation $N^{-}_{G^*}(x),N_{G^*}(x),d^{-}_{G^*}(x)$ and $d_{G^*}(x)$ is defined similarly. For any $X\subseteq V(G)$, let $G[X]$ be the subdigraph induced by $X$, and we denote $G-X$ as the digraph $G[V(G)\backslash X]$. By $|X|$ we mean the cardinality of $X$. If $X,Y\subseteq V(G)$ are disjoint, then $E(X,Y)$ denotes the set of edges from $X$ to $Y$, and denote by $|E(X,Y)|$ the number of corresponding edges. Moreover, we say $X$ dominates $Y$, denoted by $X\rightarrow Y$, if every vertex in $X$ dominates every vertex in $Y$\;(we omit the set brackets if $X$\;(resp., $Y$) contains only one vertex). For any \(S \subseteq E(G)\), let \(G\setminus S\) be the digraph with vertex set \(V(G)\) and edge set \(E(G) \setminus S\).

We shall use the following elementary observations on strong digraphs.

\begin{obser}\label{observation1}
For any strong digraph $G$, every vertex \(v\in V(G)\) satisfies that $d^+_{G}(v)\geq 1$ and $d^-_{G}(v)\geq 1$.
\end{obser}

\begin{obser}\label{observation2}
Let $G^*$ be a strong digraph. If $v$ has both out-neighbors and in-neighbors in $G^*$, then $G[V(G^*)\cup \{v\}]$ is strong.
\end{obser}

We next record two preliminary lemmas. Recall that every directed cycle in a bipartite digraph has even length.

\begin{lem}\label{lemma1}
Let $G$ be a $2k$-free strong bipartite digraph with partite sets $X$ and $Y$. Then $|X|\geq k+1$ and $|Y|\geq k+1$. Moreover, if $k=1$, then $|E(G)|\leq |X||Y|$ holds.
\end{lem}
\noindent\textbf{Proof.}
We first show the former statement. Let $C$ be a shortest directed cycle of $G$, which exists because $G$ is strong. Note that every directed cycle has even length at least $2k+2$ as $G$ is bipartite and $2k$-free. Hence $|C|\geq 2k+2$, and $C$ contains the same number of vertices in $X$ and in $Y$. This implies that $|X|\geq k+1$ and $|Y|\geq k+1$. Then we turn to the latter statement. If $k=1$, then $G$ contains no directed cycle of length two, so there is at most one edge between any $x\in X$ and $y\in Y$, which indicates that $|E(G)|\leq |X||Y|$.      
  \hfill $\blacksquare$

\begin{lem}\label{lemma2}
Let $G$ be a $2k$-free strong bipartite digraph with partite sets $X$ and $Y$. Then we have

$(1).$ for any vertex $x\in X$, there are at least $k-1$ many $y\in Y$ such that $x\nsim y$;

$(2).$ for any vertex $y\in Y$, there are at least $k-1$ many $x\in X$ such that $x\nsim y$.
\end{lem}
\noindent\textbf{Proof.}
We first prove $(1)$. For any vertex $x\in X$, we denote $Y_1=N^+_{G}(x)$ and $Y_2=N^-_{G}(x)$. Notice that $Y_1\cap Y_2$ is empty since $G$ is $2k$-free. In addition, as $G$ is strong, there must be a directed path whose initial vertex lies in $Y_1$ and whose terminal vertex lies in $Y_2$. Let $P$ be a shortest such path. Clearly, $G[V(P)\cup \{x\}]$ is an induced directed cycle of $G$ of length at least $2k+2$ satisfying that half of its vertices are in $X$ and the other are in $Y$. It follows that there are at least $k-1$ vertices $y\in Y$ such that $x\nsim y$. This proves $(1)$. By nearly identical argument one can also verify $(2)$.    \hfill $\blacksquare$

\section{Proof of Theorem~\ref{1/2}}

Let $G$ be any $4$-free bipartite digraph with bipartite sets $X$ and $Y$. We will prove Theorem~\ref{1/2} by induction on $|V(G)|=|X|+|Y|$, and we can assume that both $|X|$ and $|Y|$ are at least 2 since otherwise $G$ is acyclic, implying that $\beta(G)=0 \le \gamma(G)/2$, and we are done. Let \(f(v), g(v), h(v)\) and $i(v)$ denote the number of induced directed paths of length \(3\) where \(v\) acts as the first, second, third and last vertex, respectively. Clearly, $$\sum_{v \in V(G)} f(v) = \sum_{v \in V(G)} g(v) = \sum_{v \in V(G)} h(v) = \sum_{v \in V(G)} i(v).$$ Therefore, there exists a vertex \(v \in V(G)\) such that \(f(v) \le (g(v) + h(v))/2\). We may assume without loss of generality that \(v \in Y\). Let us denote \(X_1 = N^+_{G}(v)\), \(X_2 = N^-_{G}(v)\), and \(X_3 = X \setminus (X_1 \cup X_2)\). Moreover, we denote \(Y_1 = \{y \in Y: d^-_{G[X_1]}(y) \ge 1\}\), \(Y_2 = \{y \in Y: d^+_{G[X_2]}(y) \ge 1\}\) and \(Y_3 = Y \setminus (Y_1 \cup Y_2 \cup \{v\})\). It is evident that \(X = X_1 \cup X_2 \cup X_3\) and \(Y = Y_1 \cup Y_2 \cup Y_3 \cup \{v\}\), and we note that \(Y_1 \cap Y_2 = \emptyset\) because \(G\) is \(4\)-free. Let \(G_1\) denote the induced subdigraph of \(G\) on \(X_1 \cup Y_1\) and let \(G_2\) denote the subdigraph of \(G\) obtained from \(G\) by removing \(X_1 \cup Y_1 \cup \{v\}\).(see Figure~\ref{FIG-G-fig3})
					
					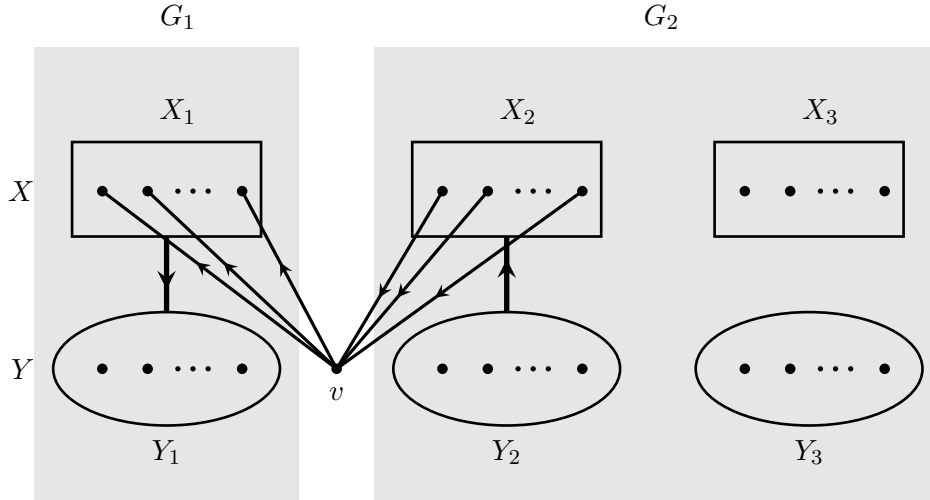
\begin{figure}[h]
						\centering
						\begin{tikzpicture}[line width=1pt,scale=0.50]
							\coordinate[label={[label distance=1mm]below:$v$}] (1) at  (-3.5,-3);
							\coordinate[] (2) at (-9.7,1.7);
							\coordinate[] (3) at (-8.5,1.7);
							\coordinate[] (4) at (-6,1.7);
							\coordinate[] (5) at (-0.7,1.7);
							\coordinate[] (6) at (0.5,1.7);
							\coordinate[] (7) at (3,1.7);
							\coordinate[] (8) at (7.3,1.7);
							\coordinate[] (9) at (8.5,1.7);
							\coordinate[] (10) at (11,1.7);
							\coordinate[] (11) at (-7.7,1.7);
							\coordinate[] (12) at (-7.3,1.7);
							\coordinate[] (13) at (-6.9,1.7);
							\coordinate[] (14) at (1.3,1.7);
							\coordinate[] (15) at (1.7,1.7);
							\coordinate[] (16) at (2.1,1.7);
							\coordinate[] (17) at (9.3,1.7);
							\coordinate[] (18) at (9.7,1.7);
							\coordinate[] (19) at (10.1,1.7);
							\coordinate[] (20) at (-9.7,-3);
							\coordinate[] (21) at (-8.5,-3);
							\coordinate[] (22) at (-6,-3);
							\coordinate[] (23) at (-0.7,-3);
							\coordinate[] (24) at (0.5,-3);
							\coordinate[] (25) at (3,-3);
							\coordinate[] (26) at (7.3,-3);
							\coordinate[] (27) at (8.5,-3);
							\coordinate[] (28) at (11,-3);
							\coordinate[] (29) at (-7.7,-3);
							\coordinate[] (30) at (-7.3,-3);
							\coordinate[] (31) at (-6.9,-3);
							\coordinate[] (32) at (1.3,-3);
							\coordinate[] (33) at (1.7,-3);
							\coordinate[] (34) at (2.1,-3);
							\coordinate[] (35) at (9.3,-3);
							\coordinate[] (36) at (9.7,-3);
							\coordinate[] (37) at (10.1,-3);
							\coordinate[label={[label distance=1mm]left:$X$}] (38) at  (-11,1.7);
							\coordinate[label={[label distance=1mm]left:$Y$}] (39) at  (-11,-3);
							
							\draw[draw=none, fill=gray!20] (-11.5,-6.5) rectangle (-4.5, 5.5)node[left=1.6cm][above=0.1cm](){$G_1$};
							\draw[draw=none, fill=gray!20] (-2.5,-6.5) rectangle (12.5, 5.5)node[left=3.7cm][above=0.1cm](){$G_2$};
							\draw[black] (-8,-3) ellipse (3 and 1.5)node[below=0.8cm](){$Y_1$};
							\draw[black] (1,-3) ellipse (3 and 1.5)node[below=0.8cm](){$Y_2$};
							\draw[black] (9,-3) ellipse (3 and 1.5)node[below=0.8cm](){$Y_3$};
							\draw[black] (-10.5,0.5) rectangle (-5.5,3) node[left=1.1cm, above=0.1cm](){$X_1$};
							\draw[black] (-1.5,0.5) rectangle (3.5, 3)node[left=1.1cm][above=0.1cm](){$X_2$};
							\draw[black] (6.5,0.5) rectangle (11.5, 3)node[left=1.1cm][above=0.1cm](){$X_3$};

							\draw[very thick,postaction={on each segment={mid arrow=black}}] (1)--(2);
							\draw[very thick,postaction={on each segment={mid arrow=black}}] (1)--(3);
							\draw[very thick,postaction={on each segment={mid arrow=black}}] (1)--(4);
							\draw[very thick,postaction={on each segment={mid arrow=black}}] (5)--(1);
							\draw[very thick,postaction={on each segment={mid arrow=black}}] (6)--(1);
							\draw[very thick,postaction={on each segment={mid arrow=black}}] (7)--(1);
							\draw[line width=2pt,postaction={on each segment={rid arrow=black}}] (-8,0.5)--(-8,-1.5);
							\draw[line width=2pt,postaction={on each segment={rid arrow=black}}] (1,-1.5)--(1,0.5);

							\node at (1) [vertice] {};
							\node at (2) [vertice] {};
							\node at (3) [vertice] {};
							\node at (4) [vertice] {};
							\node at (5) [vertice] {};
							\node at (6) [vertice] {};
							\node at (7) [vertice] {};
							\node at (8) [vertice] {};
							\node at (9) [vertice] {};
							\node at (10) [vertice] {};
							\node[scale=0.5] at (11) [vertice] {};
							\node[scale=0.5] at (12) [vertice] {};
							\node[scale=0.5] at (13) [vertice] {};
							\node[scale=0.5] at (14) [vertice] {};
							\node[scale=0.5] at (15) [vertice] {};
							\node[scale=0.5] at (16) [vertice] {};
							\node[scale=0.5] at (17) [vertice] {};
							\node[scale=0.5] at (18) [vertice] {};
							\node[scale=0.5] at (19) [vertice] {};
							\node at (20) [vertice] {};
							\node at (21) [vertice] {};
							\node at (22) [vertice] {};
							\node at (23) [vertice] {};
							\node at (24) [vertice] {};
							\node at (25) [vertice] {};
							\node at (26) [vertice] {};
							\node at (27) [vertice] {};
							\node at (28) [vertice] {};
							\node[scale=0.5] at (29) [vertice] {};
							\node[scale=0.5] at (30) [vertice] {};
							\node[scale=0.5] at (31) [vertice] {};
							\node[scale=0.5] at (32) [vertice] {};
							\node[scale=0.5] at (33) [vertice] {};
							\node[scale=0.5] at (34) [vertice] {};
							\node[scale=0.5] at (35) [vertice] {};
							\node[scale=0.5] at (36) [vertice] {};
							\node[scale=0.5] at (37) [vertice] {};
						\end{tikzpicture}
						\caption{{\footnotesize The structure of $4$-free bipartite digraph $G$.}}
						\label{FIG-G-fig3}
					\end{figure}

From the previous definitions we know that \(E(X_1, Y_3) = \emptyset\). As $G$ is $4$-free, $E(X_1,Y_2)=\emptyset$ and $E(Y_1, X_2) = \emptyset$ hold. Notice that \(f(v) = |E(Y_1, X_3)|\), that is, \(f(v)\) equals the number of edges from \(Y_1\) to \(X_3\). One can check that there exists no directed cycle passes through vertices in both \(V(G_1)\) and \(V(G_2)\) once all edges from \(Y_1\) to \(X_3\) are deleted. On the other hand, it is not difficult to deduce that $\gamma(G)\geq \gamma(G_1)+\gamma(G_2)+g(v)+h(v)$. By the induction hypothesis, we obtain
					\[
					\beta(G) \le \beta(G_1) + \beta(G_2) + f(v) \le \frac{\gamma(G_1)}{2} + \frac{\gamma(G_2)}{2} + \frac{g(v) + h(v)}{2} \le \frac{\gamma(G)}{2},
					\]
					as required.

				\section{Proof of Theorem~\ref{turan2k} }\label{sec:proof C_{2k}}

In this section we prove Theorem~\ref{turan2k}. We first establish the upper bound
\((|X|-(k-1))(|Y|-(k-1))+2k-2\) for \(2k\)-free strong bipartite digraphs with partite sets \(X\) and \(Y\), and then construct extremal examples attaining this bound.

				\subsection{Upper bound}
Let $G$ be a $2k$-free strong bipartite digraph with partite sets $X$ and $Y$. The case $k = 1$ follows from Lemma~\ref{lemma1}, so we assume $k\geq 2$. We proceed by induction on $|V(G)|=|X|+|Y|$. By Lemma~\ref{lemma1}, $|X| \ge k+1$ and $|Y| \ge k+1$. The base case $|X| =|Y| = k+1$ is immediate. In the following, assume that the assertion holds for all values smaller than $|X|+|Y|$.
					
					Let $C$ be a shortest directed cycle in $G$. Clearly, $C$ is an induced directed cycle, and we have that $|V(C)| \ge 2k+2$ since $G$ is $2k$-free. If $|V(C)| = V(G)$, then $|E(G)| = |X| + |Y|$, and the desired bound follows. Thus, we can assume that $|V(C)| \le |X| + |Y| - 1$. Let $G^*$ be a strong proper subdigraph of $G$ of maximum order. A subdigraph is said to be proper if it has strictly less than $|X|+|Y|$ vertices. Observe that the existence of $G^*$ is guaranteed by $C$. Let \(n^*=|V(G^*)|\), and let \(X^*=X\cap V(G^*)\) and \(Y^*=Y\cap V(G^*)\). Then \(2k+2\le n^*\le |X|+|Y|-1\), the digraph \(G^*\) is \(2k\)-free, and Lemma~\ref{lemma1} gives \(|X^*|,|Y^*|\ge k+1\).

 To proceed with the proof, we divide the discussion into the following three cases.
					
					 \noindent{\bf{Case \(1\).} \(n^*=n-1\).}
					
					Let $V(G) \setminus V(G^*)=\{u\}$. First suppose that $u \in X$. Then $|Y^*| = |Y|$ and $|X^*| = |X| - 1$. By applying Lemma~\ref{lemma2} we can deduce that there are at least $k-1$ many $y\in Y$ such that $u\nsim y$. This indicates that $d_{G}(u)\leq |Y| - (k-1)$. By the inductive hypothesis, we have
					\begin{align*}
						|E(G)| &= |E(G^*)| + d_{G}(u) \\
						&\le (|X^*| - (k-1))(|Y^*| - (k-1)) + 2k-2 + (|Y| - (k-1)) \\
						&= (|X| - k)(|Y| - k + 1) + 2k - 2 + |Y| - k + 1 \\
						&= (|X| - (k-1))(|Y| - (k-1)) - (|Y| - k + 1) + 2k - 2 + |Y| - k + 1 \\
						&= (|X| - (k-1))(|Y| - (k-1)) + 2k - 2,
					\end{align*}
					as required. 

By using almost identical argument we can also derive the similar conclusion when $u \in Y$.

				 \noindent	{\bf{Case \(2\).} \(n^*=n-2\).}
					
					Let \(V(G) \setminus V(G^*)=\{u,v\} \). We claim that $u\sim v$. Suppose to the contrary that $u\nsim v$. As $G$ is strong, Observation~\ref{observation1} implies that $u$ has both out-neighbors and in-neighbors, and they must lie in $V(G^*)$. It can be derived from Observation~\ref{observation2} that $G[V(G^*)\cup \{u\}]$ is a strong subdigraph of $G$. We note that $G[V(G^*)\cup \{u\}]$ is also proper because $v\notin V(G^*)\cup \{u\}$, contradicting the maximality of $G^*$. This proves that $u\sim v$. By symmetry, we can assume that $u\to v$. Then \(|X| = |X^*| + 1\) and \(|Y| = |Y^*| + 1\).

Let us denote $R=N^+_{G^*}(v)$ and $S=N^-_{G^*}(u)$. Note that both $R$ and $S$ are nonempty because $G$ is strong, and both $N^-_{G^*}(v)$ and $N^+_{G^*}(u)$ are empty by the maximality of $G^*$. Let \(P\) be a shortest directed path in \(G^*\) from a vertex of \(R\) to a vertex of \(S\). Obviously, $G[V(P)\cup \{u,v\}]$ is an induced directed cycle of even length greater than $2k$. It follows that the sum of $d_{G}(u)$ and $d_{G}(v)$ is at most $|X|+|Y| -2k+2$. Consequently, we can obtain that
					\[
					\begin{aligned}
						|E(G)| &= |E(G^*)|+d_{G}(u)+d_{G}(v)-1 \\
						&\le (|X^*| - (k-1))(|Y^*| - (k-1)) + 2k-2 +|X|+|Y|-2k+2-1 \\
						&= (|X| - k)(|Y| - k) +|X|+|Y|-1 \\
						&= |X||Y|-(k-1)|X|-(k-1)|Y|+k^2 -1 \\
						&= (|X|-(k-1))(|Y|-(k-1)) + 2k-2,
					\end{aligned}
					\]
					as required.
					
	 \noindent				{\bf{Case \(3\).} \(n^*\le n-3 \).}

Since $G$ is strong, there must be a vertex $u\in V(G)\setminus V(G^*)$ such that $N^-_{G^*}(u)\neq \emptyset$, and there also exists a directed path starting at $u$ and terminating at some vertex of $G^*$. For convenience, we denote $Q=u_1 u_2 \dots u_\ell$ to be the directed subpath of the previous one. By the maximality of $G^*$, we can derive that $V(Q)=V(G)\setminus V(G^*)$. Moreover, for every $v\in V(G^*)$ and every $i\in[2,\ell-1]$, $v\nsim u_i$, and  $N^+_{G^*}(u_1)=N^-_{G^*}(u_\ell)=\emptyset$. Symmetrically, we can assume that $u_1\in X$, and hence $u_i\in X$\;(resp., $u_i\in Y$) for odd\;(resp., even) $i\in[\ell]$. Let $X_0$ and $Y_0$ denote the sets of vertices of $Q$ with odd and even indices. It is easily seen that \(X=X^*\cup X_0\) and \(Y=Y^*\cup Y_0\), and \(|X_0| - |Y_0|\in\{0,1\}\). We shall consider the following two subcases.
					
					{\bf{Subcase \(3.1\).} \(|X_0|=|Y_0|+1 \).}
					
Let $\ell=2t-1$. Then $t\geq 2$, $|X_0|=t$ and $|Y_0|=t-1$. This yields that \(|X| = |X^*| + t\) and \(|Y| = |Y^*| + t - 1\). By the inductive hypothesis, we get
					
		\[
					\begin{aligned}
						\gamma(G^*) &\ge |X^*||Y^*|-((|X^*| - (k-1))(|Y^*| - (k-1)) +2k-2) \\
						&=|X^*||Y^*|-(|X^*||Y^*|-(k-1)|X^*|-(k-1)|Y^*|+(k-1)^2)-2k+2 \\
						&= (k-1)(|X^*| + |Y^*|) - k^2+1.
					\end{aligned}
					\]
Combining this with the previous results gives
					\[
					\begin{aligned}
						\gamma(G) &\ge (k-1)(|X^*| + |Y^*|) - k^2 +1+(t-1)|X^*| +(t-2)|Y^*|+2(k-1) \\
						&= (k-1)(|X| + |Y|-2t+1) +(t-1)|X^*| +(t-2)|Y^*|- k^2 +2k-1 .
					\end{aligned}
					\]
					Hence
					\[
					\begin{aligned}
						&\gamma(G) - ((k-1)(|X| + |Y|) - k^2+1) \\
						\ge\ &(k-1)(1-2t) +(t-1)(k+1) + (t - 2)(k+1) +2k-2 \\
						=\ &4t - 6 > 0.
					\end{aligned}
					\]
We can thus derive that \(\gamma(G) > (k-1)(|X| + |Y|) - k^2+1\), implying that
$$|E(G)| = |X||Y| - \gamma(G) \le (|X| - (k-1))(|Y| - (k-1)) + 2k-2.$$
					
					{\bf{Subcase \(3.2\).} \(|X_0|=|Y_0|\).}
					
					Let \(\ell= 2t\). Then $t\geq 2$, \(|X_0| = |Y_0|=t \). In addition,  \(|X| = |X^*| + t\) and \(|Y| = |Y^*| + t \) hold. 
As in the previous subcase, we also have that
					$$\gamma(G^*) \geq (k-1)(|X^*| + |Y^*|) - k^2+1,$$
and
					\[
					\begin{aligned}
						\gamma(G) &\ge (k-1)(|X^*| + |Y^*|) - k^2 +1+(t-1)(|X^*| +|Y^*|)+2(k-1)-1 \\
						&=(k-1)(|X|+|Y|-2t)+(t-1)(|X^*| +|Y^*|)- k^2+2k-2.
					\end{aligned}
					\]
	Then the following holds
					\[
					\begin{aligned}
						&\gamma(G) - ((k-1)(|X| + |Y|) - k^2+1) \\
						\ge\ &-2kt+2t+(t-1)(2k+2)+2k-3 \\
						=\ &4t - 5 > 0.
					\end{aligned}
					\]
Consequently, we conclude that \(|E(G)| = |X||Y| - \gamma(G) \le (|X| - (k-1))(|Y| - (k-1)) + 2k-2 \), as required.
					
\subsection{Lower bound}

We now construct a family of extremal bipartite digraphs as below.

\vspace{0.2cm}
\noindent{\bf Construction of $\mathcal{F}_1$:} Let $C=\langle x_{0},y_{0},x_1,y_1,\ldots,x_k,y_k\rangle$ be a directed cycle of length $2k+2$. Let $\mathcal{F}_1$ be a family of digraphs obtained from $C$ by blowing up its vertices $x_{0}$ and $y_0$ into independent sets $X^*$ and $Y^*$, respectively. That is, $X^*\to Y^*$, $Y^*\to x_1$ and $y_k\to X^*$.\;(see Figure~\ref{FIG-G-fig1})
\vspace{0.2cm}

Observe that every member $G$ of $\mathcal{F}_1$ is a $2k$-free strong bipartite digraph with bipartite sets $X = X^* \cup \{x_1, x_2, \dots, x_k\}$ and $Y = Y^* \cup \{y_1, y_2, \dots, y_k\}$, where $|X| \ge k+1$ and $|Y| \ge k+1$. Its number of edges is
			\begin{align*}
				|E(G)| &= |X^*||Y^*| + |X^*| + |Y^*| + 2k - 1 \\
				&= (|X^*| + 1)(|Y^*| + 1) + 2k - 2 \\
				&= (|X| - (k-1))(|Y| - (k-1)) + 2k - 2.
			\end{align*}

	\begin{figure}[t]
				\centering
				\begin{tikzpicture}[line width=1pt,scale=0.50]
					\coordinate[] (1) at (-1.7,3);
					\coordinate[] (2) at (-0.5,3);
					\coordinate[] (3) at (2,3);
					\coordinate[] (4) at (-1.7,-3);
					\coordinate[] (5) at (-0.5,-3);
					\coordinate[] (6) at (2,-3);
					\coordinate[label={[label distance=1mm]above:$x_1$}] (7) at (5,3);
					\coordinate[label={[label distance=1mm]above:$x_2$}] (8) at (7,3);
					\coordinate[label={[label distance=1mm]above:$x_3$}] (9) at (9,3);
					\coordinate[] (10) at (10,3);
					\coordinate[] (11) at (11,3);
					\coordinate[] (12) at (12,3);
					\coordinate[label={[label distance=1mm]above:$x_{k-1}$}] (13) at (13,3);
					\coordinate[label={[label distance=1mm]below:$y_1$}] (14) at (5,-3);
					\coordinate[label={[label distance=1mm]below:$y_2$}] (15) at (7,-3);
					\coordinate[label={[label distance=1mm]below:$y_3$}] (16) at (9,-3);
					\coordinate[] (17) at (10,-3);
					\coordinate[] (18) at (11,-3);
					\coordinate[] (19) at (12,-3);
					\coordinate[label={[label distance=1mm]below:$y_{k-1}$}] (20) at (13,-3);
					\coordinate[] (22) at (0.7,3);
					\coordinate[] (21) at (0.3,3);
					\coordinate[] (23) at (1.1,3);
					\coordinate[] (24) at (0.3,-3);
					\coordinate[] (25) at (0.7,-3);
					\coordinate[] (26) at (1.1,-3);
					\coordinate[label={[label distance=1mm]above:$x_k$}] (27) at (15,3);
					\coordinate[label={[label distance=1mm]below:$y_k$}] (28) at (15,-3);
					
					\draw[black] (0,3) ellipse (3 and 1.5)node[left=1.5cm](){$X^*$};
					\draw[black] (0,-3) ellipse (3 and 1.5)node[left=1.5cm](){$Y^*$};
					\draw[very thick,postaction={on each segment={rid arrow=black}}] (7)--(14);
					\draw[very thick,postaction={on each segment={lid arrow=black}}] (14)--(8);
					\draw[very thick,postaction={on each segment={mid arrow=black}}] (8)--(15);
					\draw[very thick,postaction={on each segment={mid arrow=black}}] (15)--(9);
					\draw[very thick,postaction={on each segment={mid arrow=black}}] (9)--(16);
					\draw[very thick,postaction={on each segment={mid arrow=black}}] (13)--(20);
					\draw[very thick,postaction={on each segment={mid arrow=black}}] (20)--(27);
					\draw[very thick,postaction={on each segment={mid arrow=black}}] (27)--(28);
					\draw[line width=2pt,postaction={on each segment={mid arrow=black}}] (0,1.5)--(0,-1.5);
					\draw[line width=2pt,postaction={on each segment={lid arrow=black}}] (0,-1.5)--(7);
					\draw[line width=2pt,postaction={on each segment={lid arrow=black}}] (28)--(0,1.5);

					\node at (1) [vertice] {};
					\node at (2) [vertice] {};
					\node at (3) [vertice] {};
					\node at (4) [vertice] {};
					\node at (5) [vertice] {};
					\node at (6) [vertice] {};
					\node at (7) [vertice] {};
					\node at (8) [vertice] {};
					\node at (9) [vertice] {};
					\node[scale=0.5] at (10) [vertice] {};
					\node[scale=0.5] at (11) [vertice] {};
					\node[scale=0.5] at (12) [vertice] {};
					\node at (13) [vertice] {};
					\node at (14) [vertice] {};
					\node at (15) [vertice] {};
					\node at (16) [vertice] {};
					\node[scale=0.5] at (17) [vertice] {};
					\node[scale=0.5] at (18) [vertice] {};
					\node[scale=0.5] at (19) [vertice] {};
					\node at (20) [vertice] {};
					\node[scale=0.5] at (21) [vertice] {};
					\node[scale=0.5] at (22) [vertice] {};
					\node[scale=0.5] at (23) [vertice] {};
					\node[scale=0.5] at (24) [vertice] {};
					\node[scale=0.5] at (25) [vertice] {};
					\node[scale=0.5] at (26) [vertice] {};
					\node at (27) [vertice] {};
					\node at (28) [vertice] {};
				\end{tikzpicture}
				\caption{{\footnotesize The construction of $\mathcal{F}_1$.}}
				\label{FIG-G-fig1}
			\end{figure}
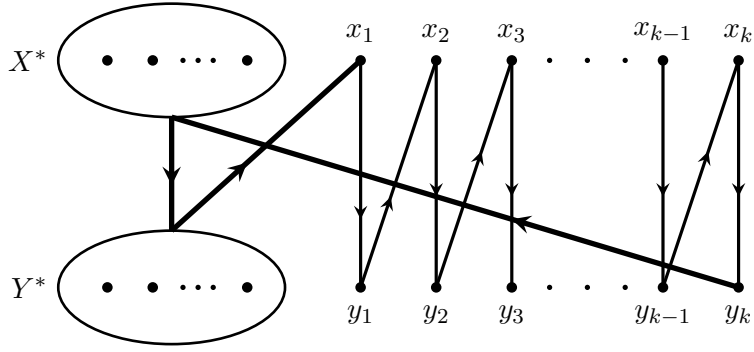

Combining the above results, we thus complete our proof of Theorem~\ref{turan2k}.

				\section{Proof of Theorem~\ref{C_{2k}}}\label{sec:proof C_4}

This section is devoted to the proof of Theorem~\ref{C_{2k}}. We first record a structural description of \(4\)-free strong bipartite digraphs in which one part has exactly three vertices.

\vspace{0.2cm}
\noindent{\bf Construction of $\mathcal{F}_2$:} Let $C=\langle x_{1},y_{1},x_2,y_2,x_3,y_3\rangle$ be a directed cycle of length $6$. Let $\mathcal{F}_2$ be a family of digraph obtained from $C$ by blowing up its vertices $y_{1},y_2$ and $y_3$ into independent sets $Y_1,Y_2$ and $Y_3$, respectively. That is, $x_i\to Y_i$ and $Y_i\to x_{i+1}$ for any $i\in [3]$.\;(see Figure~\ref{FIG-G-fig2})
\vspace{0.2cm}

		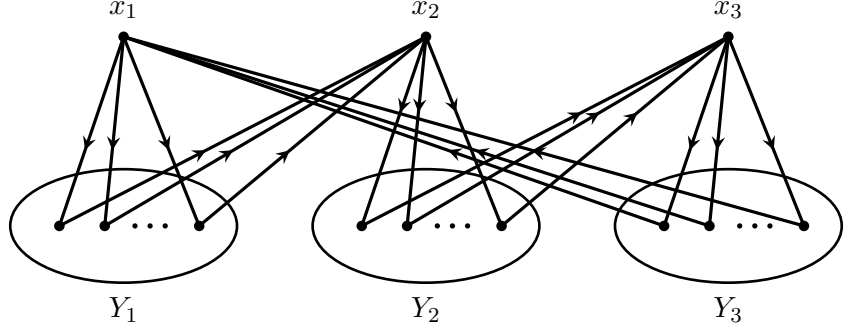
\begin{figure}[h]
			\centering
			\begin{tikzpicture}[line width=1pt,scale=0.50]
				\coordinate[label={[label distance=1mm]above:$x_1$}] (1) at (-8,2);
				\coordinate[label={[label distance=1mm]above:$x_2$}] (2) at (0,2);
				\coordinate[label={[label distance=1mm]above:$x_3$}] (3) at (8,2);
				\coordinate[] (4) at (-9.7,-3);
				\coordinate[] (5) at (-8.5,-3);
				\coordinate[] (6) at (-6,-3);
				\coordinate[] (7) at (-1.7,-3);
				\coordinate[] (8) at (-0.5,-3);
				\coordinate[] (9) at (2,-3);
				\coordinate[] (10) at (6.3,-3);
				\coordinate[] (11) at (7.5,-3);
				\coordinate[] (12) at (10,-3);
				\coordinate[] (13) at (-7.7,-3);
				\coordinate[] (14) at (-7.3,-3);
				\coordinate[] (15) at (-6.9,-3);
				\coordinate[] (16) at (0.3,-3);
				\coordinate[] (17) at (0.7,-3);
				\coordinate[] (18) at (1.1,-3);
				\coordinate[] (19) at (8.3,-3);
				\coordinate[] (20) at (8.7,-3);
				\coordinate[] (21) at (9.1,-3);
				
				\draw[black] (-8,-3) ellipse (3 and 1.5)node[left=0cm][below=0.8cm](){$Y_1$};
				\draw[black] (0,-3) ellipse (3 and 1.5)node[left=0cm][below=0.8cm](){$Y_2$};
				\draw[black] (8,-3) ellipse (3 and 1.5)node[left=0cm][below=0.8cm](){$Y_3$};
				\draw[very thick,postaction={on each segment={mid arrow=black}}] (1)--(4);
				\draw[very thick,postaction={on each segment={mid arrow=black}}] (1)--(5);
				\draw[very thick,postaction={on each segment={mid arrow=black}}] (1)--(6);
				\draw[very thick,postaction={on each segment={lid arrow=black}}] (2)--(7);
				\draw[very thick,postaction={on each segment={lid arrow=black}}] (2)--(8);
				\draw[very thick,postaction={on each segment={lid arrow=black}}] (2)--(9);
				\draw[very thick,postaction={on each segment={mid arrow=black}}] (3)--(10);
				\draw[very thick,postaction={on each segment={mid arrow=black}}] (3)--(11);
				\draw[very thick,postaction={on each segment={mid arrow=black}}] (3)--(12);
				\draw[very thick,postaction={on each segment={lid arrow=black}}] (4)--(2);
				\draw[very thick,postaction={on each segment={lid arrow=black}}] (5)--(2);
				\draw[very thick,postaction={on each segment={lid arrow=black}}] (6)--(2);
				\draw[very thick,postaction={on each segment={mid arrow=black}}] (7)--(3);
				\draw[very thick,postaction={on each segment={mid arrow=black}}] (8)--(3);
				\draw[very thick,postaction={on each segment={mid arrow=black}}] (9)--(3);
				\draw[very thick,postaction={on each segment={lid arrow=black}}] (10)--(1);
				\draw[very thick,postaction={on each segment={lid arrow=black}}] (11)--(1);
				\draw[very thick,postaction={on each segment={lid arrow=black}}] (12)--(1);

				\node at (1) [vertice] {};
				\node at (2) [vertice] {};
				\node at (3) [vertice] {};
				\node at (4) [vertice] {};
				\node at (5) [vertice] {};
				\node at (6) [vertice] {};
				\node at (7) [vertice] {};
				\node at (8) [vertice] {};
				\node at (9) [vertice] {};
				\node at (10) [vertice] {};
				\node at (11) [vertice] {};
				\node at (12) [vertice] {};
				\node[scale=0.5] at (13) [vertice] {};
				\node[scale=0.5] at (14) [vertice] {};
				\node[scale=0.5] at (15) [vertice] {};
				\node[scale=0.5] at (16) [vertice] {};
				\node[scale=0.5] at (17) [vertice] {};
				\node[scale=0.5] at (18) [vertice] {};
				\node[scale=0.5] at (19) [vertice] {};
				\node[scale=0.5] at (20) [vertice] {};
				\node[scale=0.5] at (21) [vertice] {};
			\end{tikzpicture}
			\caption{{\footnotesize The construction of $\mathcal{F}_2$.}}
			\label{FIG-G-fig2}
		\end{figure}

Let $F$ be any $4$-free strong bipartite digraph with partite sets $X$ and $Y$, where $|X|=3$. We will briefly illustrate that $F\in \mathcal{F}_2$. Observe first that $|Y|\geq 3$ because $F$ is strong and any directed cycle in it has length greater than 4. As $|X|=3$, we know that any directed cycle of $F$ has length exactly 6. Let $C=\langle x_{1},y_{1},x_2,y_2,x_3,y_3\rangle$ be an arbitrary directed cycle of $F$. For any $y\in Y\setminus \{y_1,y_2,y_3\}$, both $d^+_{F}(y)$ and $d^-_{F}(y)$ are positive, and $N_{F}(y)\subseteq X$. If $x_i\to y$, then $y\nrightarrow x_{i-1}$ since $F$ is $4$-free. Moreover, we can check that $N_{F}(y)\neq X$, which implies that $y\to x_{i+1}$ and $y\nsim x_{i-1}$. This shows that $F\in \mathcal{F}_2$.  
A simple calculation gives us that
		$$|E(F)| = 2|Y| = (|X|-1)|Y| = (|X|-1)(|Y|-1) + 2.$$

Now we are ready to prove Theorem~\ref{C_{2k}}.

\vspace{0.2cm}

\noindent\textbf{Proof of Theorem~\ref{C_{2k}}.} Let $G$ be any $4$-free strong bipartite Tur\'{a}n digraph with bipartite sets $X$ and $Y$. Note that both $|X|$ and $|Y|$ are larger than two, and by Theorem~\ref{turan2k}}, $$|E(G)|=(|X|-1)(|Y|-1)+2,$$ which tells us that $\gamma(G)=|X|+|Y|-3$. We will prove this theorem by induction on $|V(G)|=|X|+|Y|$. By symmetry, we may assume that $|Y|\geq |X|$. If $|X|=3$, then $G$ is a member of $\mathcal{F}_2$. Writing $X=\{x_1,x_2,x_3\}$ and $Y=Y_1\cup Y_2 \cup Y_3$, we have $$ \beta(G)=\min\{|Y_1|,|Y_2|,|Y_3|\}\;\;\;\text{and}\;\;\;\gamma(G)=|Y_1|+|Y_2|+|Y_3|.$$ Hence, we infer that $\beta(G)\leq\frac13\gamma(G)$.
			
		In the following, we may assume that \(|Y|\ge |X|\ge 4 \), and that the assertion holds for all values smaller than $|X|+|Y|$. Choose $y\in Y$ with minimum degree in $Y$. We first claim that \(d_G(y)\le |X|-2\). Suppose to the contrary that \(d_G(y)\ge |X|-1\). Then \(d_G(y')\ge |X|-1\) for every \(y'\in Y\). Thereby, we obtain that $$|E(G)|=\sum_{y'\in Y} d_G(y')\ge (|X|-1)|Y|.$$ Together with \(|E(G)|=(|X|-1)(|Y|-1)+2\), we thus derive that \((|X|-1)(|Y|-1)+2\ge (|X|-1)|Y|\), which follows that \(|X|\le 3\), a contradiction. This proves \(d_G(y)\le |X|-2\).

Next, let us show that \(G-y\) is not strong. Suppose for sake of contradiction that \(G-y\) is strong. We note that $G-y$ is a $4$-free strong bipartite digraph with partite sets $X$ and $Y\setminus\{y\}$. By exploiting Theorem~\ref{turan2k}, we have that $$|E(G-y)|\le (|X|-1)(|Y|-2)+2.$$ Using  \(d_G(y)\le |X|-2\), we obtain $$|E(G)|=|E(G-y)|+d_G(y)\le (|X|-1)(|Y|-2)+2+(|X|-2)=(|X|-1)(|Y|-1)+1,$$ contrary to the extremality of $G$. Hence \(G-y\) is not strong.

Let $G_{1}, G_{2},\ldots,G_{s}$ be the strong components of $G-y$, where $s\geq 2$, satisfying that there is no edge from $V(G_j)$ to $V(G_{i})$, where $1\leq i<j\leq s$. Note that each $G_i$ contains either a unique vertex or at least $6$ vertices. 
	For each $i\in[s]$, write
\[
X_i:=V(G_i)\cap X
\qquad\text{and}\qquad
Y_i:=V(G_i)\cap Y.
\]
We further denote
\[
X^*:=X_2\cup\ldots\cup X_{s-1}
\qquad\text{and}\qquad
Y^*:=Y_2\cup\ldots\cup Y_{s-1}.
\]
Then
\[
X=X_1\cup X^*\cup X_s
\qquad\text{and}\qquad
Y=Y_1\cup Y^*\cup Y_s\cup \{y\}.
\]

Clearly, $ N^{+}_{G}(y)\cap X_{1}\neq \emptyset$ and $ N^{-}_{G}(y)\cap X_{s}\neq \emptyset$ because $G$ is strong. We write $G_{i}\cup \{y\}$ for $G[V(G_i)\cup \{y\}]$, and $E(G_i,G_j)$ for $E(V(G_i),V(G_j))$, where $i,j\in[s]$. In particular, we denote \(G^*=G[X^* \cup Y^*]\). We will divide the proof into the following two parts.

\begin{center}
\noindent \textbf{Part I. $G-y$ has exactly two strong components $G_1$ and $G_2$.}
\end{center}

		Then \( X=X_1\cup X_2\) and \(Y=Y_1\cup Y_2\cup \{y\}\).
					We shall consider the following four cases.
					
					\smallskip
					\noindent
					\textbf{Case 1. Both $G_1\cup\{y\}$ and $G_2\cup\{y\}$ are strong.}
					
		By Observation~\ref{observation1}, there must exist vertices \(x_1^*\in X_1\) and \(x_2^*\in X_2\) such that \(yx_1^*\in E(G)\) and \(x_2^*y\in E(G)\). Note that for each $i\in[2]$, $G_i$ is also strong, and both $X_i$ and $Y_i$ contain at least three vertices. Again by Observation~\ref{observation1} we know that $x_1^*$ has at least one out-neighbor $y_1^*\in Y_1$ and $x_2^*$ has at least one in-neighbor $y_2^*\in Y_2$. Since $G$ is $4$-free, it is clear that neither \(x_1^*y_2^*\) nor \(y_1^*x_2^*\) can be an edge. Thereby $|E(X_1,Y_2)|\leq |X_1||Y_2|-1$ and $|E(Y_1,X_2)|\leq |X_2||Y_1|-1$, yielding that $$|E(G_1,G_2)|\le |X_1||Y_2|+|X_2||Y_1|-2.$$ On the other hand, applying Theorem~\ref{turan2k}, we can deduce that \(|E(G_1\cup\{y\})|\le (|X_1|-1)|Y_1|+2\) and \(|E(G_2\cup\{y\})|\le (|X_2|-1)|Y_2|+2\). Therefore, we have
					\[
					\begin{aligned}
						|E(G)|
						&= |E(G_1\cup\{y\})|+|E(G_2\cup\{y\})|+|E(G_1,G_2)|\\
						&\le (|X_1|-1)|Y_1|+2+(|X_2|-1)|Y_2|+2 +|X_1||Y_2|+|X_2||Y_1|-2\\
						&=(|X|-1)(|Y|-1)+2.
					\end{aligned}
					\]

Since $G$ is a bipartite Tur\'{a}n digraph and equality holds in the preceding bound, one can see that \(|E(G_1\cup\{y\})|=(|X_1|-1)|Y_1|+2\), \(|E(G_2\cup\{y\})|=(|X_2|-1)|Y_2|+2\) and  \(|E(G_1,G_2)|= |X_1||Y_2|+|X_2||Y_1|-2\). Indeed, \(|E(X_1,Y_2)|=|X_1||Y_2|-1\) and \(|E(Y_1,X_2)|=|X_2||Y_1|-1\) hold. From the above argument, we can further derive that $y$ has a unique out-neighbor in $X_1$, and also a unique in-neighbor in $X_2$, that is,  \(N^+_{G_1}(y)=\{x_1^*\}\) and \(N^-_{G_2}(y)=\{x_2^*\}\).

On the other hand, by Lemma~\ref{lemma2}, there exists a vertex $u\in X_1$\;(resp., $w\in X_2$) such that $y\nsim u$\;(resp., $y\nsim w$). It yields that $d_{G_1}(y)\leq |X_1|-1$ and $d_{G_2}(y)\leq |X_2|-1$. Together with $|E(G_i)|\leq (|X_i|-1)
(|Y_i|-1)+2$, we thus get$$|E(G_i\cup\{y\})|\leq (|X_i|-1)(|Y_i|-1)+2+|X_i|-1=(|X_i|-1)|Y_i|+2,$$
which implies that $d_{G_i}(y)=|X_i|-1$ and $|E(G_i)|=(|X_i|-1)(|Y_i|-1)+2$, where $i\in[2]$. As a consequence, we obtain that both $G_i$ and $G_i\cup \{y\}$ are bipartite Tur\'{a}n digraph for any $i\in[2]$. From the inductive hypothesis, one easily infers \(\beta(G_i)\le \frac13\gamma(G_i)\) and \(\beta(G_i\cup\{y\})\le \frac13\gamma(G_i\cup\{y\})\) for every $i\in [2]$.

Every directed cycle of \(G\) is either contained in \(G_1\cup\{y\}\), contained in \(G_2\), or uses the edge \(x_2^*y\). It is not difficult to see that every directed cycle of $G\setminus x_2^*y$ is contained entirely in either \(G_1\cup\{y\}\) or \(G_2\). Hence, we get \(\beta(G)\le \beta(G_1\cup\{y\})+\beta(G_2)+1\).

In summary, we obtain
					\[
					\begin{aligned}
						\beta(G)
						&\le \beta(G_1\cup\{y\})+\beta(G_2)+1\\
                                &\le \frac13\gamma(G_1\cup\{y\})+\frac13\gamma(G_2)+1\\
						&= \frac13(|X_1|+|Y_1|-2)+\frac13(|X_2|+|Y_2|-3)+1\\
						&= \frac13(|X|+|Y|-3)= \frac13\gamma(G).
					\end{aligned}
					\]
					
					\smallskip
					\noindent
					\textbf{Case 2. $G_1\cup\{y\}$ is strong, whereas $G_2\cup\{y\}$ is not strong.}
					
					\smallskip
					\noindent
					\textbf{(2.1) \(G_2\) is nontrivial.}
					
			Similarly, there is a vertex \(x_1^*\in X_1\)\;(resp., $y_1^*\in Y_1$) such that \(yx_1^*\in E(G)\)\;(resp., $x_1^*y_1^*\in E(G)$). Since \(G\) is \(4\)-free, every directed cycle in \(G_2\) has length at least \(6\), which means that \(|X_2|\ge 3\) and \(|Y_2|\ge 3\). Additionally, \(|E(G_1\cup\{y\})|\le (|X_1|-1)|Y_1|+2\) and \(|E(G_2)|\le (|X_2|-1)(|Y_2|-1)+2\). Let us denote $W=N^-_{G_2}(y)\subseteq X_2$. The set $W$ is nonempty since $G$ is strong, and \(N^+_{G_2}(y) = \emptyset\) since \(G_2 \cup \{y\}\) is not strong.
			For every \(w\in W\), there exists a vertex \(w^*\in Y_2\) such that \(w^*w\in E(G)\). Since \(G\) is \(4\)-free, we have \(x_1^*\nrightarrow w^*\).
			Hence $|E(X_1,Y_2)|\leq |X_1||Y_2|-1$. Meanwhile, it is evident that $E(y_1^*, W)=\emptyset$, which follows that $|E(Y_1,X_2)|\leq |X_2||Y_1|-|W|$.
					
By the above results, we thus deduce that 
					\[
					\begin{aligned}
						|E(G)|
						&= |E(G_1\cup\{y\})|+|E(G_2)|+|E(G_2,y)|+|E(G_1,G_2)|\\
						&\le (|X_1|-1)|Y_1|+2+(|X_2|-1)(|Y_2|-1)+2+|W|+(|X_1||Y_2|-1)+(|X_2||Y_1|-|W|)\\
						&=(|X|-1)(|Y|-1)-|X_2|+4\\
						&\leq (|X|-1)(|Y|-1)+1\\
                                &< (|X|-1)(|Y|-1)+2,
					\end{aligned}
					\]
					contradicting the fact that $G$ is a bipartite Tur\'{a}n digraph.
					
					\smallskip
					\noindent
					\textbf{(2.2) \(G_2\) is a singleton component.}
					
					One easily checks that $G_2$ contains a unique vertex which lies in \(X\). Thus \(|X_2|=1\) and \(|Y_2|=0\). The similar counting argument gives
					\[
					\begin{aligned}
						|E(G)|
						&=|E(G_1\cup\{y\})|+|E(G_2,y)|+|E(G_1,G_2)|\\
						&\le (|X_1|-1)|Y_1|+2+1+(|Y_1|-1)\\
						&=|X_1||Y_1|+2\\
						&=(|X|-1)(|Y|-1)+2.
					\end{aligned}
					\]

Recall that $G$ is a bipartite Tur\'{a}n digraph, equality yields \(|E(G_1\cup\{y\})|=(|X_1|-1)|Y_1|+2\).  From Lemma~\ref{lemma2} we know that there exists a vertex $u\in X_1$ with $y\nsim u$, and hence $d_{G_1}(y)\leq |X_1|-1$. Combining this with $|E(G_1)|\leq (|X_1|-1)
(|Y_1|-1)+2$, one can derive that$$|E(G_1\cup\{y\})|\leq (|X_1|-1)(|Y_1|-1)+2+|X_1|-1=(|X_1|-1)|Y_1|+2,$$ 
yielding that $d_{G_1}(y)=|X_1|-1$ and $|E(G_1)|=(|X_1|-1)(|Y_1|-1)+2$. Hence, we conclude that \(d_G(y)=d_{G_1}(y)+1=|X_1|=|X|-1\), contradicting the assumption that \(d_G(y)\le |X|-2\).
					
					\smallskip
					\noindent
					\textbf{Case 3. $G_1\cup\{y\}$ is not strong, but $G_2\cup\{y\}$ is strong.}

			We can deduce, by applying almost identical method as that of {\textbf{Case 2}}, that such $G$ also does not exist.

					\smallskip
					\noindent
					\textbf{Case 4. Neither \(G_1\cup\{y\}\) nor \(G_2\cup\{y\}\) is strong.}
					
					Notice first that at least one of \(G_1\) and \(G_2\) is nontrivial. We distinguish the following three subcases.
					
					\smallskip
					\noindent
					\textbf{(4.1) Both \(G_1\) and \(G_2\) are nontrivial.}
					
					Since \(G_i\) is a nontrivial strong bipartite digraph, we get $|X_i|\geq 3$, and by the inductive hypothesis, we further have \(|E(G_i)|\le (|X_i|-1)(|Y_i|-1)+2\), where $i\in[2]$. 
					
					It can be derived from Observation~\ref{observation2} that \(y\) cannot have both out-neighbors and in-neighbors in the same \(G_i\) because \(G_i\cup\{y\}\) is not strong, where \(i\in[2]\). Moreover, since \(G\) is strong, there must exist an edge from \(y\) to \(G_1\) and an edge from \(G_2\) to \(y\). Thus every edge incident with \(y\) is either directed from \(y\) to \(X_1\), or from \(X_2\) to \(y\), implying that \(d_G(y)=d^+_{G_1}(y)+d^-_{G_2}(y)\). Let us denote $W_1=N^+_{G_1}(y)$ and $W_2=N^-_{G_2}(y)$. As both $G_1$ and $G_2$ are strong, there exists a vertex $w_1^+\in Y_1$ which is dominated by some vertex of $W_1$, and there exists a vertex $w_2^-$ which dominates some vertex of $W_2$. It is clear that there is no edge from $W_1$ to $w_2^-$, and no edge from $w_1^+$ to $W_2$ because $G$ is $4$-free. This indicates that \(|E(X_1,Y_2)|\le |X_1||Y_2|-d^+_{G_1}(y)\) and \(|E(Y_1,X_2)|\le |X_2||Y_1|-d^-_{G_2}(y)\). Thereby \(|E(G_1,G_2)|\le |X_1||Y_2|+|X_2||Y_1|-d^+_{G_1}(y)-d^-_{G_2}(y)\). By combining these estimates, we get
					\[
					\begin{aligned}
						|E(G)|
						&\le |E(G_1)|+|E(G_2)|+|E(G_1,G_2)|+d_G(y)\\
						&\le (|X_1|-1)(|Y_1|-1)+2+(|X_2|-1)(|Y_2|-1)+2
						+|X_1||Y_2|+|X_2||Y_1|\\
						&=(|X|-1)(|Y|-1)-|X|+6\\
                                &\leq (|X|-1)(|Y|-1)\\
                                &< (|X|-1)(|Y|-1)+2,
					\end{aligned}
					\]
					 contradicting the fact that \(G\) is a bipartite Tur\'{a}n digraph.
					
					\smallskip
					\noindent
					\textbf{(4.2) \(G_1\) is nontrivial, whereas \(G_2\) is a singleton component.}
					
					Similarly, one can see that \(|E(G_1)|\le (|X_1|-1)(|Y_1|-1)+2\). The component $G_2$ consists of a single vertex, say \(x_2\), and $x_2\to y$. Note that \(X=X_1\cup \{x_2\}\) and \(Y=Y_1\cup \{y\}\). Pick arbitrarily a vertex \(x_1^*\in X_1\) such that \(yx_1^*\in E(G)\). Applying Observation~\ref{observation1}, there must be a vertex \(y_1^*\in Y_1\) with \(x_1^*y_1^*\in E(G)\). Since \(G\) is \(4\)-free, \(y_1^*\nrightarrow x_2\). This means that \(|E(G_1,G_2)|=|E(Y_1,X_2)|\le |Y_1|-1\). Combining this with \(d_G(y)\le |X|-2\), we can obtain
					\[
					\begin{aligned}
						|E(G)|
						&\le |E(G_1)|+|E(G_1,G_2)|+d_G(y)\\
						&\le (|X_1|-1)(|Y_1|-1)+2+(|Y_1|-1)+(|X|-2)\\
						&= |X_1||Y_1|-|X_1|+|X|\\
						&=(|X|-1)(|Y|-1)+1\\
						&< (|X|-1)(|Y|-1)+2,
					\end{aligned}
					\]
					also a contradiction.
					
					\smallskip
					\noindent
					\textbf{(4.3) \(G_2\) is nontrivial, whereas \(G_1\) is a singleton component.}
					
					This subcase is symmetric to \textbf{(4.2)}, and we omit its proof here.
					
					\medskip
					This completes the proof of {\textbf{Part I}}.

\begin{center}
\noindent \textbf{Part II. \(G-y\) has at least three strong components.}
\end{center}

					Observe first that \(G_1\) is an initial component and
					\(G_s\) is a  terminal component in \(G-y\), where $s\geq 3$. Since \(G\) is strong, there exist
					\(x_1^*\in X_1\) and \(x_s^*\in X_s\) such that \(yx_1^*\in E(G)\) and
					\(x_s^*y\in E(G)\). For convenience, we denote $R_1=N^+_{G^*}(y)$ and $R_2=N^-_{G^*}(y)$. Then $R_1\cap R_2=\emptyset$ and $R_1\cup R_2\subseteq X^*$.

To proceed with the proof, we will divide the discussion into the following four cases.

					\smallskip
					\noindent
					\textbf{Case 1. Both \(G_1\cup\{y\}\) and \(G_s\cup\{y\}\) are strong.}
					
			In this case, it is easily seen that $|X_i|, |Y_i|\geq 3$ for each $i\in \{1,s\}$. By using Observation~\ref{observation1}, there must be \(y_1^*\in Y_1\) and \(y_s^*\in Y_s\) such that \(x_1^*y_1^*\in E(G)\) and \(y_s^*x_s^*\in E(G)\). We note that \(x_1^*y_s^*\notin E(G)\) and \(y_1^*x_s^*\notin E(G)\). Hence \(|E(X_1,Y_s)|\le |X_1||Y_s|-1\), \(|E(Y_1,X_s)|\le |Y_1||X_s|-1\), and therefore $|E(G_1,G_s)|\leq |X_1||Y_s|+|Y_1||X_s|-2$. 

Additionally, as $G$ is $4$-free, one checks that for any vertex $r\in R_1$\;(resp., $r\in R_2$), $ry_s^*\notin E(G)$\;(resp., $y_1^*r\notin E(G)$). For any $y^*\in Y^*$, the two edges \(x_1^*y^*\) and \(y^*x_s^*\) cannot both appear since otherwise \(\langle y,x_1^*,y^*,x_s^*\rangle\) would be a directed $4$-cycle. This gives us that $$|E(G_1,G^*)|+|E(G^*,G_s)|\leq |X_1||Y^*|+|X^*||Y_1|+|X^*||Y_s|+|X_s||Y^*|-|R_1|-|R_2|-|Y^*|.$$  Combining with \(|E(G^*)|\le |X^*||Y^*|\), we get
					\[
					\begin{aligned}
						|E(G)|
						&= |E(G_1\cup\{y\})|+|E(G_s\cup\{y\})|+|E(G^*)|+|E(G_1,G_s)|+|E(G_1,G^*)|+|E(G^*,G_s)|+d_{G^*}(y)\\
						&\le (|X_1|-1)|Y_1|+2+(|X_s|-1)|Y_s|+2+|X^*||Y^*|+|X_1||Y_s|+|Y_1||X_s|-2\\
                                &\qquad +|X_1||Y^*|+|X^*||Y_1|+|X^*||Y_s|+|X_s||Y^*|-|Y^*|\\
						&= (|X_1|+|X^*|+|X_s|)(|Y_1|+|Y^*|+|Y_s|)-|Y_1|-|Y^*|-|Y_s|+2\\
						&= |X|(|Y|-1)-|Y|+3\\
						&= (|X|-1)(|Y|-1)+2.
					\end{aligned}
					\]

By the condition $G$ is a bipartite Tur\'{a}n digraph, equality holds in the preceding estimates.
We conclude that \(|E(G_1\cup\{y\})|=(|X_1|-1)|Y_1|+2\), \(|E(G_s\cup\{y\})|=(|X_s|-1)|Y_s|+2\),  \(|E(G_1,G_s)|= |X_1||Y_s|+|X_s||Y_1|-2\) and $|E(G^*)|= |X^*||Y^*|$.  More precisely, one can infer that \(|E(X_1,Y_s)|=|X_1||Y_s|-1\) and \(|E(Y_1,X_s)|=|X_s||Y_1|-1\),  and \(N^+_{G_1}(y)=\{x_1^*\}\) and \(N^-_{G_s}(y)=\{x_s^*\}\). According to $|E(G^*)|=|X^*||Y^*|$ and $G$ is $4$-free, we know that $G^*\cup \{y\}$ must be acyclic.

On the other hand, by applying Lemma~\ref{lemma2}, there exists a vertex $u\in X_1$\;(resp., $w\in X_s$) with $y\nsim u$\;(resp., $y\nsim w$). It follows that $d_{G_1}(y)\leq |X_1|-1$ and $d_{G_s}(y)\leq |X_s|-1$. Combining this with $|E(G_i)|\leq (|X_i|-1)
(|Y_i|-1)+2$, one sees that$$|E(G_i\cup\{y\})|\leq (|X_i|-1)(|Y_i|-1)+2+|X_i|-1=(|X_i|-1)|Y_i|+2,$$
yielding that $d_{G_i}(y)=|X_i|-1$ and $|E(G_i)|=(|X_i|-1)(|Y_i|-1)+2$, where $i\in \{1,s\}$. Thus both $G_i$ and $G_i\cup \{y\}$ are bipartite Tur\'{a}n digraphs for each $i\in \{1,s\}$. By the inductive hypothesis, \(\beta(G_i)\le \frac13\gamma(G_i)\) and \(\beta(G_i\cup\{y\})\le \frac13\gamma(G_i\cup\{y\})\) for every $i\in \{1,s\}$.

Notice that any directed cycle of \(G\) is either contained in \(G_1\), contained in \(G_s\), or contains one of the edges \(yx_1^*\) and \(x_s^*y\). It is not hard to deduce that every directed cycle of $G\setminus \{yx_1^*,x_s^*y\}$ is contained entirely in either \(G_1\) or \(G_s\).

If $R_1=\emptyset$, then
					\[
					\begin{aligned}
						\beta(G)
                                &\le \beta(G_1)+\beta(G_s\cup\{y\})+1\\
						&\le \frac13\gamma(G_1)+\frac13\gamma(G_s\cup\{y\})+1\\
						&= \frac13(|X_1|+|Y_1|-3)+\frac13(|X_s|+|Y_s|-2)+1\\
						&= \frac13(|X_1|+|X_s|+|Y_1|+|Y_s|-2)\\
						&\le \frac13(|X|+|Y|-3)
						= \frac13\gamma(G).
					\end{aligned}
					\]
					
Similarly, if $R_2=\emptyset$, then $$\beta(G)
						\le \beta(G_1\cup \{y\})+\beta(G_s)+1
						\leq\frac13\gamma(G).$$		

					
We next turn to the case that both $R_1$ and $R_2$ are nonempty. It is clear that \(|X^*|\geq |R_1\cup R_2|\ge 2\), and we thus have $|X|\geq |X_1|+|X_s|+2$. 
Thereby, the following holds		\[
					\begin{aligned}
						\beta(G)
                                &\leq \beta(G_1)+\beta(G_s)+2\\
						&\le \frac13\gamma(G_1)+\frac13\gamma(G_s)+2\\
						&= \frac13(|X_1|+|Y_1|-3)+\frac13(|X_s|+|Y_s|-3)+2\\
                                &\le \frac13(|X|+|Y|-3)=\frac13\gamma(G).
					\end{aligned}
					\]

					\smallskip
					\noindent
					\textbf{Case 2. \(G_1\cup\{y\}\) is strong, whereas \(G_s\cup\{y\}\) is not strong.}
					
					
					\smallskip
					\noindent
					\textbf{(2.1) \(G_s\) is nontrivial.}
					
					Clearly, both \(G_1\) and \(G_s\) are strong. By using Observation~\ref{observation1}, there must be \(y_1^*\in Y_1\) such that \(x_1^*y_1^*\in E(G)\). Since \(G_s\cup\{y\}\) is not strong but \(G\) is strong, we thus have \(N^+_{G_s}(y)=\emptyset\) by Observation~\ref{observation2}. By the inductive hypothesis we have \(|E(G_1\cup\{y\})|\le (|X_1|-1)|Y_1|+2\) and \(|E(G_s)|\le (|X_s|-1)(|Y_s|-1)+2\). 
					
			Denote by $W=N^-_{G_s}(y)$ and $x_s^*\in W$. Note that there is a vertex $y_s^*\in Y_s$ such that $y_s^*\to x_s^*$. As $G$ is $4$-free, one checks that \(x_1^*y_s^*\notin E(G)\) and there is no edge from $y_1^*$ to any vertex of $W$. This indicates that \(|E(X_1,Y_s)|\le |X_1||Y_s|-1\) and \(|E(Y_1,X_s)|\le |X_s||Y_1|-|W|\).
In addition, it is not difficult to see that for any vertex $r\in R_1$\;(resp., $r\in R_2$), $ry_s^*\notin E(G)$\;(resp., $y_1^*r\notin E(G)$). This means that \(|E(Y_1,X^*)|\le |X^*||Y_1|-|R_2|\) and  \(|E(X^*,Y_s)|\le |X^*||Y_s|-|R_1|\). For any $y^*\in Y^*$, if \(x_1^*y^*\in E(G)\), then $E(y^*,W)=\emptyset$. Hence, \(|E(X_1,Y^*)|+|E(Y^*,X_s)|\le (|X_1|+|X_s|-1)|Y^*|\). Then 
$$|E(G_1,G^*)|+|E(G^*,G_s)|\leq (|X_1|+|X_s|-1)|Y^*|+|X^*|(|Y_1|+|Y_s|)-|R_1|-|R_2|.$$ 

By combining the above results with \(|E(G^*)|\le |X^*||Y^*|\), we conclude that
					\[
					\begin{aligned}
						|E(G)|
						&= |E(G_1\cup\{y\})|+|E(G_s)|+|E(G^*)|+|E(G_1,G_s)|
						+|E(G_1,G^*)|+|E(G^*,G_s)|+d_{G^*\cup G_s}(y)\\
						&\le (|X_1|-1)|Y_1|+2+(|X_s|-1)(|Y_s|-1)+2+|X^*||Y^*|+|X_1||Y_s|+|X_s||Y_1|-1\\
						&\qquad +|X_1||Y^*|+|X_s||Y^*|-|Y^*|
						+|X^*||Y_1|+|X^*||Y_s|\\
						&= (|X_1|+|X^*|+|X_s|-1)(|Y_1|+|Y^*|+|Y_s|)
						-|X_s|+4\\
						&= (|X|-1)(|Y|-1)-|X_s|+4.                  
				\end{aligned}
					\]
					Since \(G_s\) is a nontrivial strong \(4\)-free bipartite digraph, we have \(|X_s|\ge 3\). As a consequence, we get
					\[
					|E(G)|\le (|X|-1)(|Y|-1)+1<(|X|-1)(|Y|-1)+2,
					\]
					which contradicts the fact that \(G\) is a bipartite Tur\'{a}n digraph.

					\smallskip
					\noindent
					\textbf{(2.2) \(G_s\) is a singleton component.}
					
					Here \(X_s=\{x_s^*\}\) and \(Y_s=\emptyset\), and \(x_s^*y\in E(G)\).
					Similarly, there
					exists \(y_1^*\in Y_1\) such that \(x_1^*y_1^*\in E(G)\). Since
					\(G_1\cup\{y\}\) is strong, one sees that
					\(|E(G_1\cup\{y\})|\le (|X_1|-1)|Y_1|+2\).
					
					We next show that \(|E(G^*)|\le |X^*||Y^*|-|R_1|\). It suffices to take the case that \(R_1\neq\emptyset\) because otherwise we                                     
                         are done.  For any
					\(x\in R_1\), choose a shortest directed path
					\(P_1=xy_1'x_1'y_2'\ldots y\) from \(x\) to \(y\) avoiding \(G_1\). Such a
					path must exist as otherwise the acyclic ordering of the strong components
					of \(G-y\) would be violated. Then \(P_1\cup \{yx\}\) is a directed cycle of length 
					 at least \(6\) because \(G\) is \(4\)-free. Hence
					\(y_2'\in Y^*\). From the minimality of \(P_1\) we know that \(xy_2'\notin E(G)\), and from $G$ is $4$-free we get
					 \(y_2'x\notin E(G)\).					
					This indicates that \(x\nsim y_2'\), and consequently
					\(|E(G^*)|\le |X^*||Y^*|-|R_1|\).
					
					Again using the fact that $G$ is 4-free, we can obtain \(|E(G_1,G_s)|=|E(Y_1,x_s^*)|\le |Y_1|-1\), and for every
					\(y^*\in Y^*\), the two edges \(x_1^*y^*\) and \(y^*x_s^*\) cannot both
					occur. Then
					\(|E(X_1,Y^*)|+|E(Y^*,x_s^*)|\le |X_1||Y^*|\). Similarly, for every
					\(x^*\in R_2\), \(y_1^*x^*\notin E(G)\), so
					 \(|E(Y_1,X^*)|\le |X^*||Y_1|-|R_2|\). 

                          Therefore, the following holds
					\[
					\begin{aligned}
						|E(G)|
						&= |E(G_1\cup\{y\})|+|E(G^*)|+|E(G_1,G_s)|
						+|E(G_1,G^*)|+|E(G^*,G_s)|+d_{G^*\cup G_s}(y)\\
						&\le (|X_1|-1)|Y_1|+2+|X^*||Y^*|+|Y_1|-1 +|X_1||Y^*|+|X^*||Y_1|+1\\
						&=(|X_1|+|X^*|)(|Y_1|+|Y^*|)+2\\
						&=(|X|-1)(|Y|-1)+2.
					\end{aligned}
					\]
					Since \(G\) is a bipartite Tur\'an digraph, we conclude that \(|E(G_1\cup\{y\})|=(|X_1|-1)|Y_1|+2\),
					\(|E(G^*)|=|X^*||Y^*|-|R_1|\), \(|E(Y_1,x_s^*)|=|Y_1|-1\),
					\(|E(X_1,Y^*)|+|E(Y^*,x_s^*)|=|X_1||Y^*|\) and
					\(|E(Y_1,X^*)|=|X^*||Y_1|-|R_2|\).
					
					We now verify that the inductive hypothesis can be applied to the
					nontrivial components appearing below. Since
					\(|E(G_1\cup\{y\})|=(|X_1|-1)|Y_1|+2\), by Lemma~\ref{lemma2}, there is a vertex
					\(u\in X_1\) such that \(y\nsim u\), yielding that \(d_{G_1}(y)\le |X_1|-1\).
					By Theorem~\ref{turan2k}, we know \(|E(G_1)|\le (|X_1|-1)(|Y_1|-1)+2\). Thus
					\[
					|E(G_1\cup\{y\})|
					\le (|X_1|-1)(|Y_1|-1)+2+|X_1|-1
					=(|X_1|-1)|Y_1|+2.
					\]
					It follows that \(|E(G_1)|=(|X_1|-1)(|Y_1|-1)+2\), which means \(G_1\) is
					a bipartite Tur\'an digraph and \(d_{G_1}(y)= |X_1|-1\).

Moreover, let us prove that $d^+_{G_1}(y)=1$ and $d^-_{G_1}(y)=|X_1|-2$. Suppose not, there are two vertices $\hat{x}_1,\hat{x}_2\in X_1$ with $y\to \hat{x}_1$ and $y\to \hat{x}_2$. As \(|E(Y_1,x_s^*)|=|Y_1|-1\), there is a unique vertex $\hat{y^*}$ satisfying that $\hat{x}_i\to \hat{y^*}$ for $i=1,2$. If $Y^*=\emptyset$, then $X^*$ is not empty because $s\geq 3$. We claim that any vertex of $X^*$ dominates $y$. Otherwise, 
we are able to find a vertex $\hat{x}^*$ of $X^*$ such that $\hat{x}^*\nrightarrow y$. Then there is a directed path from $\hat{x}^*$ to $y$ in $G[V(G^*)\cup \{y,x_s^*\}]$, contradicting $Y^*=\emptyset$. This shows that every vertex of $X^*$ dominates $y$. Together with \(d_{G_1}(y)= |X_1|-1\), we thus have $d_{G}(y)= |X|-1$, contradicting the fact that $d_{G}(y)\leq |X|-2$. Therefore, we can assume that $Y^*$ is not empty.
There must be a directed path of length two starting at some vertex of $Y^*$ and terminating at $y$. This indicates that
					\[
					|E(X_1,Y^*)|+|E(Y^*,x_s^*)|
					\le |X_1||Y^*|-1,
					\]
which is a contradiction. Consequently, we conclude that $d^+_{G_1}(y)=1$ and $d^-_{G_1}(y)=|X_1|-2$.

					Next, we will verify that there is no nontrivial strong component for any $i$ with
					\(2\le i\le s-1\). Suppose to the contrary that there exists a nontrivial strong component for some $i\in[2,s-1]$. If \(R_1=\emptyset\), then \(|E(G^*)|=|X^*||Y^*|\), so there is no missing edge between \(X^*\) and \(Y^*\). This indicates that no nontrivial strong component exists for any $i$ with $i\in [2,s-1]$, which leads to a contradiction.
					Now we turn to the case that \(R_1\neq\emptyset\). Notice that for every vertex $\hat{x}$ of $X_i$, $d_{G_i}(\hat{x})\leq |Y_i|-1$, namely, there is some $\hat{y}\in Y_i$ such that $\hat{x}\nsim \hat{y}$. Together with the fact that \(|E(G^*)|=|X^*||Y^*|-|R_1|\), we can deduce that $\hat{x}\in R_1$, and thus $X_i\subseteq R_1$. Let $P_2$ be a shortest directed path with initial vertex in $X_i$ and terminal vertex $y$. It is apparent that such $P_2$ must exist, and $P_2\cup y$ is a directed cycle of length at least 6. One easily sees that the initial vertex of $P_2$ is not adjacent to some vertex of $Y^*\setminus Y_i$. Thereby, we infer that \(|E(G^*)|\leq |X^*||Y^*|-|R_1|-1\), contrary to \(|E(G^*)|=|X^*||Y^*|-|R_1|\). This proves that any strong component $G_i$ contains a unique vertex, where $2\leq i\leq s-1$.

					Moving on, we claim that $G^*\cup \{y\}$ is acyclic. Observe first that $G^*$ is acyclic. It is enough to show that $G^*$ contains no directed path from \(R_1\) to \(R_2\). Suppose such a path exists, and let \(P_3=x^+y_1'x_1'y_2'\ldots x^-\) be a
					shortest one, where \(x^+\in R_1\) and \(x^-\in R_2\). Then
					\(P_3\cup\{x^-y,yx^+\}\) is a directed cycle containing \(y\), and its length is at least \(6\) as $G$ is \(4\)-free. It is clear that $x^-$ is not adjacent to some vertex of $Y^*$, which yields that  
					\(|E(G^*)|\leq |X^*||Y^*|-|R_1|-1\), a contradiction. Hence no directed path  from \(R_1\) to \(R_2\) exists
in \(G^*\), and so $G^*\cup \{y\}$ is acyclic.


The proceding structural analysis shows that every directed cycle of $G[V(G^*)\cup \{x_s^*,y\}]$ contains the edge $x_s^*y$. By the choice of $y$, the analogous argument also gives $Y^*\neq \emptyset$. Now we are ready to verify that $\beta(G)\leq \frac13\gamma(G)$.

					
					
					If \(R_1=\emptyset\), then  deleting the edge \(yx_1^*\), together with a feedback
					edge set of $G_1$, destroys all directed cycles of
					\(G\). Therefore, we have
					\[
					\begin{aligned}
						\beta(G)
						&\le \beta(G_1)+1\\
						&\le \frac13\gamma(G_1)+1\\
						&=\frac13(|X_1|+|Y_1|-3)+1\\
						&\leq\frac13(|X|+|Y|-3)=\frac13\gamma(G).
					\end{aligned}
					\]
					
					If \(R_1\neq\emptyset\), then the 4-freeness of $G$, together with the existence of a directed path from every vertex of $R_1$ to $y$, implies $|X^*|\geq 2$ and $|Y^*|\geq 2$.  Deleting the edges \(yx_1^*\) and \(x_s^*y\),
					together with a feedback edge set of \(G_1\), destroys all directed
					cycles of \(G\). Hence, we obtain that
					\[
					\begin{aligned}
						\beta(G)
						&\le \beta(G_1)+2\\
						&\le \frac13\gamma(G_1)+2\\
						&= \frac13(|X_1|+|Y_1|-3)+2\\
						&\leq \frac13(|X|+|Y|-9)+2\\
						&= \frac13(|X|+|Y|-3)=\frac13\gamma(G).
					\end{aligned}
					\]
					

					\smallskip
					\noindent
					\textbf{Case 3. \(G_1\cup\{y\}\) is not strong, whereas \(G_s\cup\{y\}\) is strong.}
					
					By an argument analogous to that used in \textbf{Case 2}, we can prove that $\beta(G)\leq \frac13\gamma(G)$. We omit its proof  here.

					\smallskip
					\noindent
					\textbf{Case 4. Neither \(G_1\cup\{y\}\) nor \(G_s\cup\{y\}\) is strong.}
					
					Since \(G_1\) is an initial component and \(G_s\) is a terminal component of \(G-y\), we have \(d^+_{G_1}(y)>0\) and \(d^-_{G_s}(y)>0\). Additionally, \(d^-_{G_1}(y)=0\) and
					\(d^+_{G_s}(y)=0\) as neither \(G_1\cup\{y\}\) nor \(G_s\cup\{y\}\) is strong.
					
					\smallskip
					\noindent
					\textbf{(4.1) Both \(G_1\) and \(G_s\) are nontrivial.}
					
					 By Theorem~\ref{turan2k},
					\(|E(G_1)|\le (|X_1|-1)(|Y_1|-1)+2\) and
					\(|E(G_s)|\le (|X_s|-1)(|Y_s|-1)+2\) hold. Pick \(x_1^*\in N^+_{G_1}(y)\) and \(x_s^*\in N^-_{G_s}(y)\). By
					Observation~\ref{observation1}, there exist \(y_1^*\in Y_1\) and
					\(y_s^*\in Y_s\) such that \(x_1^*y_1^*\in E(G)\) and
					\(y_s^*x_s^*\in E(G)\). Since \(G\) is \(4\)-free, for every
					\(x\in N^+_{G_1}(y)\), the edge \(xy_s^*\) is absent, and for every \(x\in N^-_{G_s}(y)\), the edge \(y_1^*x\) is absent.
					Then
					\[
					|E(G_1,G_s)|
					\le |X_1||Y_s|+|X_s||Y_1|-d^+_{G_1}(y)-d^-_{G_s}(y).
					\]
					
					 For
					every \(x\in R_1\)\;(resp., \(x\in R_2\)), \(xy_s^*\notin E(G)\)\;(resp., \(y_1^*x\notin E(G)\)). Thus
					\(|E(Y_1,X^*)|\le |X^*||Y_1|-|R_2|\) and
					\(|E(X^*,Y_s)|\le |X^*||Y_s|-|R_1|\). Moreover, for every \(y^*\in Y^*\), the edges
					\(x_1^*y^*\) and \(y^*x_s^*\) cannot both appear since otherwise
					\(\langle y,x_1^*,y^*,x_s^*\rangle\) would be a directed 4-cycle.
					Thereby,
					\(
					|E(X_1,Y^*)|+|E(Y^*,X_s)|
					\le (|X_1|+|X_s|-1)|Y^*|.
					\)
					Combining the above estimates, we get
					\[
						|E(G_1,G^*)|+|E(G^*,G_s)|
						\le |X^*||Y_1|+|X^*||Y_s|+(|X_1|+|X_s|-1)|Y^*|-|R_1|-|R_2| .
					\]
					
					Together with \(d_G(y)=d^+_{G_1}(y)+d^-_{G_s}(y)+|R_1|+|R_2|\), we can obtain
					\[
					\begin{aligned}
						|E(G)|
						&= |E(G_1)|+|E(G_s)|+|E(G^*)|
						+|E(G_1,G_s)|+|E(G_1,G^*)|+|E(G^*,G_s)|+d_G(y)\\
						&\le (|X_1|-1)(|Y_1|-1)+2
						+(|X_s|-1)(|Y_s|-1)+2+|X^*||Y^*|\\
						&\qquad +|X_1||Y_s|+|X_s||Y_1|
						+|X^*||Y_1|+|X^*||Y_s|
						+(|X_1|+|X_s|-1)|Y^*|\\
						&=(|X_1|+|X^*|+|X_s|-1)(|Y_1|+|Y^*|+|Y_s|)
						-|X_1|-|X_s|+6\\
						&=(|X|-1)(|Y|-1)-|X_1|-|X_s|+6.
					\end{aligned}
					\]
					Note that \(|X_1|\ge3\) and \(|X_s|\ge3\).
					We can deduce that \(|E(G)|\le (|X|-1)(|Y|-1)\), contradicting the fact that \(G\) is a
					bipartite Tur\'an digraph.
					
					\smallskip
					\noindent
					\textbf{(4.2) \(G_1\) is nontrivial, whereas \(G_s\) is a singleton component.}
					
					Clearly, \(X_s=\{x_s^*\}\) and \(Y_s=\emptyset\), and \(x_s^*y\in E(G)\).
					From Observation~\ref{observation2} we know that \(d^-_{G_1}(y)=0\), and by Theorem~\ref{turan2k}, we have
					\(|E(G_1)|\le (|X_1|-1)(|Y_1|-1)+2\).
					As $G_1$ is strong, there must be a vertex \(y_1^*\in Y_1\) such that
					\(x_1^*y_1^*\in E(G)\). Moreover, $y_1^*x_s^*\notin E(G)$ as $G$ is 4-free.
					
					Let us show that \(Y^*\neq\emptyset\). Assume the statement is not true and hence \(Y^*=\emptyset\). All edges of \(G-y\) outside \(G_1\) are contained in \(E(Y_1,X\setminus X_1)\). As $s\geq 3$, it is easy to see that \(X^*\neq\emptyset\). It is not hard to check that any vertex of $X^*$ dominates $y$ since $G$ is strong and there is no edge from $G_i$ to $G_1$ for any  $i\in[2,s]$. Again by using $G$ is $4$-free, we can infer that $y_1^*\nrightarrow x$ for every $x\in X^*$. By combining with $d_G(y)\leq |X|-2$, we have the following
					\[
					\begin{aligned}
						|E(G)|
						&=|E(G_1)|+|E(Y_1,X\setminus X_1)|+d_G(y)\\
						&\le (|X_1|-1)(|Y_1|-1)+2+(|X|-|X_1|)(|Y_1|-1)+(|X|-2)\\
						&=(|X|-1)(|Y|-1)+1,
					\end{aligned}
					\]
					a contradiction. This proves that \(Y^*\neq\emptyset\).
					
					 Write $W$ for $N^+_{G_1}(y)$.  Obviously, $x_1^*\in W$, and thus $|W|\geq 1$.  We would like to prove that $ |E(X_1,Y^*)|+|E(Y^*,x_s^*)|\le |X_1||Y^*|-|W|+1$. Denote by $\eta=|E(Y^*,x_s^*)|$. We first deal with $\eta\geq 1$. For every \(y'\in Y^*\) with \(y'x_s^*\in E(G)\) and every
					\(x\in W\), the edge \(xy'\) is absent as otherwise
					\(\langle y,x,y',x_s^*\rangle\) would be a directed \(4\)-cycle. Hence, we get
					\(|E(X_1,Y^*)|\le |X_1||Y^*|-\eta |W|\), and so
					\[
					|E(X_1,Y^*)|+|E(Y^*,x_s^*)|
					\le |X_1||Y^*|-\eta |W|+\eta\le |X_1||Y^*|-|W|+1,
					\]
	as desired. Now we turn to $\eta=0$. Choose any vertex $\hat{y}$ of \(Y^*\). Since \(G\)
					is strong, there must be a directed path from \(\hat{y}\) to $y$. Pick arbitrarily a shortest one say $P$. The penultimate vertex of $P$ belongs to $X^*$ since $\eta=0$. Denote this vertex by $\tilde{x}$ and denote $\tilde{y}$ by its predecessor. Observe that $\tilde{y}\in Y^*$ and there is no edge from $W$ to $\tilde{y}$. This implies that
					\(|E(X_1,Y^*)|+|E(Y^*,x_s^*)|\le |X_1||Y^*|-|W|\). In summary, we conclude that
					\[
					|E(X_1,Y^*)|+|E(Y^*,x_s^*)|
					\le |X_1||Y^*|-|W|+1.
					\]

By similar argument as that of \textbf{(2.2)}, we can derive that \(|E(G^*)|\le |X^*||Y^*|-|R_1|\).
					Moreover, since \(G\) is \(4\)-free, we have \(|E(Y_1,x_s^*)|\le |Y_1|-1\), and for
					every \(x\in R_2\), the edge \(y_1^*x\) is absent, yielding that
					\(|E(Y_1,X^*)|\le |X^*||Y_1|-|R_2|\). Therefore
					\[
					|E(G_1,G^*)|+|E(G^*,G_s)|
					\le |X_1||Y^*|-|W|+1+|X^*||Y_1|-|R_2|.
					\]
					Recall that \(x_s^*y\in E(G)\) and \(d^+_{G_s}(y)=0\). Then
					\(d_G(y)=|W|+|R_1|+|R_2|+1\). 
					
From the above results one can derive that
\[
					\begin{aligned}
						|E(G)|
						&= |E(G_1)|+|E(G^*)|+|E(G_1,G_s)|
						+|E(G_1,G^*)|+|E(G^*,G_s)|+d_G(y)\\
						&\le (|X_1|-1)(|Y_1|-1)+2+|X^*||Y^*|+|Y_1|-1
						+|X_1||Y^*|+1+|X^*||Y_1|+1\\
						&=(|X_1|+|X^*|)(|Y_1|+|Y^*|)-|X_1|+4\\
						&=(|X|-1)(|Y|-1)-|X_1|+4.
					\end{aligned}
					\]
					Since \(G_1\) is nontrivial, we know that \(|X_1|\ge3\). Consequently, we can deduce
					\(|E(G)|\le (|X|-1)(|Y|-1)+1\), which contradicts the fact that \(G\) is a
					bipartite Tur\'an digraph.
					
					\smallskip
					\noindent
					\textbf{(4.3) \(G_1\) is a singleton component, whereas \(G_s\) is nontrivial.}
					
					We can also prove that such $G$ does not exist by nearly identical arguments as that of \textbf{(4.2)}.
					
					\smallskip
					\noindent
					\textbf{(4.4) Both \(G_1\) and \(G_s\) are singleton components.}
					
					Since \(G_1\) and \(G_s\) are singleton components and \(G\) is strong,
					their unique vertices lie in \(X\). Thus \(X_1=\{x_1^*\}\),
					\(X_s=\{x_s^*\}\), and \(yx_1^*,x_s^*y\in E(G)\).
					
					Next we would like to prove that every $G_i$ is a singleton component for every $i\in [2,s-1]$. Suppose, on the contrary, that there is a nontrivial \(G_i\) for some $i$ with $2\leq i\leq s-1$. We note first that both $X_i$ and $Y_i$ contain at least three vertices. By Theorem~\ref{turan2k}, it is obvious that $|E(G_i)|\leq (|X_i|-1)(|Y_i|-1)+2$, namely, there are at least \(|X_i|+|Y_i|-3\) pairs of nonadjacent vertices with both ends in $V(G_i)$. For every \(y'\in Y\setminus\{y\}\), the two edges \(x_1^*y'\) and
					\(y'x_s^*\) cannot appear simultaneously as otherwise
					\(\langle x_s^*,y,x_1^*,y'\rangle\) would be a directed 4-cycle. This means there are at least
					\(|Y|-1\) pairs of nonadjacent vertices with one end in $\{x_1^*,x_s^*\}$ and another in $Y$. From Lemma~\ref{lemma2} one can deduce that any vertex of $X$ has degree at most $|Y|-1$. Particularly, every vertex of \(X\setminus (X_i\cup\{x_1^*,x_s^*\})\) has degree at most $|Y|-1$.
By combining the previous results, we conclude that
					\[
					\begin{aligned}
						|E(G)|
						&\le |X||Y|-\bigl((|X_i|+|Y_i|-3)+(|Y|-1)+(|X|-|X_i|-2)\bigr)\\
						&= |X||Y|-|X|-|Y|-|Y_i|+6.
					\end{aligned}
					\]
Recall that \(G\) is a bipartite Tur\'an digraph and $|Y_i|\geq 3$. One can get \(|Y_i|=3\). In particular, every vertex of \(X_i\) is adjacent to \(y\).
			Otherwise, we can find another pair of nonadjacent vertices which is different from all the above, leading to a contradiction.
					
					We next show that all edges between \(y\) and \(X_i\) have the same
					orientation. If not, then $G_i\cup \{y\}$ is a strong subdigraph of $G$ which is also $4$-free and $y$ has degree $|X_i|$ in $G_i\cup \{y\}$. However, it is impossible by Lemma~\ref{lemma2}. We can assume without loss of generality that \(yx\in E(G)\) for every \(x\in X_i\).
					
					Since \(G\) is strong, there is a directed path from \(G_i\) to $y$. Let $P=\hat{x}_1\hat{y}_1\ldots \hat{x}_r\hat{y}_r$ be such a directed path of minimum length, where $\hat{y}_r=y$. Notice that $V(P)\cap X_i=\hat{x}_1$ and $P\cup (y,\hat{x}_1)$ is a directed cycle, we thus have $r\geq 3$ since $G$ is 4-free. It is evident that  $\hat{y}_{r-1}\notin Y_i$ and $\hat{x}_1\nsim \hat{y}_{r-1}$. That is, we are able to find one more pair of nonadjacent vertices. This indicates that the size of $G$ is strictly less than $(|X|-1)(|Y|-1)+2$, which is a contradiction. This proves that for any $i\in [2,s-1]$, every $G_i$ is a singleton component.
					
					Obviously, \(G-y\) is acyclic, and every directed cycle of \(G\) passes through \(y\).
					By deleting all edges with \(y\) as head\;(resp., tail) we can destroy all
					directed cycles in $G$. Therefore
					\[
					\beta(G)\le \min\{d_G^+(y),d_G^-(y)\}\le \frac12d_G(y)\le \frac12(|X|-2).
					\]
				Together with \(|Y|\ge |X|\), we thus deduce that
					\[
					\beta(G)\le \frac12(|X|-2)\le \frac14(|X|+|Y|-4)
					<\frac13(|X|+|Y|-3)=\frac13\gamma(G).
					\]

					
				\vspace{0.2cm}
					This completes the proof of \textbf{Part II}.

	\vspace{0.2cm}
					Combining the proofs of \textbf{Part I} and \textbf{Part II}, we complete the proof of
Theorem~\ref{C_{2k}}.  \hfill $\blacksquare$

\section{Concluding Remarks}

This paper studies feedback edge sets in bipartite digraphs from the perspective of short directed cycles. For a bipartite digraph \(G\), we let \(\gamma(G)\) count only nonadjacent pairs with ends in distinct partite sets. With this normalization, Theorem~\ref{1/2} gives a general estimate \(\beta(G)\le\gamma(G)/2\) for every \(4\)-free bipartite digraph. We also determined the exact Tur\'an number of \(2k\)-free strong bipartite digraphs with prescribed partite sets, and then used the extremal structure in the case \(k=2\) to prove \(\beta(G)\le\gamma(G)/3\) for all \(4\)-free strong bipartite Tur\'an digraphs.

These results suggest that the correct constant in the first bipartite case should be \(1/3\), not \(1/2\). More generally, they lead to the following possible bipartite analogue of the Chudnovsky--Seymour--Sullivan \cite{chudnovsky2008cycles}\;(3-free) and Sullivan \cite{sullivan2008extremal}\;($k$-free) feedback-edge problems.

\begin{conj}\label{k.2k}
	If \(G\) is a \(2k\)-free bipartite digraph for \(k\ge2\), then
	\[
	\beta(G)\le \frac{1}{k^2-1}\gamma(G).
	\]
\end{conj}

The first nontrivial case is \(k=2\), where directed \(4\)-cycles are the first new even obstruction beyond directed \(2\)-cycles.

\begin{conj}\label{k=2}
	If \(G\) is a \(4\)-free bipartite digraph, then
	\[
	\beta(G)\le \frac13\gamma(G).
	\]
\end{conj}

Theorem~\ref{C_{2k}} verifies Conjecture~\ref{k=2} for the densest possible strong examples, and the balanced three-block construction shows that the coefficient \(1/3\) cannot be improved in that class. The main remaining problem is to remove the Tur\'an extremality assumption and prove, or disprove, Conjecture~\ref{k=2} for all \(4\)-free bipartite digraphs.

Two further directions seem natural. First, it would be interesting to obtain stability or structural results for \(4\)-free bipartite digraphs whose edge number is close to the Tur\'an bound, and to determine whether such near-extremal structure already forces the coefficient \(1/3\). Second, Seymour and Spirkl \cite{seymour2020short} proposed several conjectures relating short directed cycles in bipartite digraphs to minimum out-degree conditions. Since Conjecture~\ref{1.2} is closely connected with feedback-edge inequalities through Conjecture~\ref{1.1}, it would be worthwhile to investigate whether the bipartite feedback-edge perspective developed here can be linked to that short-cycle program.

\subsection*{Acknowledgement}

\noindent

The work was supported by the National Natural Science Foundation of China (Nos. 12471336, 12501473).

\vspace{0.3cm}

%
%


				%
				%

				

			\end{document}